\RequirePackage{fix-cm}

\documentclass[smallextended,nospgthms]  {svjour3}

\smartqed 
\usepackage[latin1]{inputenc}
\usepackage{hyperref}
\usepackage[headinclude,DIV13]{typearea}
\areaset{15.1cm}{25.0cm}
\usepackage{amsfonts,amssymb,amsmath,bbm}
\usepackage{cancel} 
\usepackage{graphicx}
\usepackage{caption}
\usepackage{subcaption}

\usepackage{xcolor}
\usepackage{enumerate}
\usepackage{dsfont}

\spnewtheorem*{notations}{Notation.}{\itshape}{\rmfamily}
\newtheorem{assumption}[theorem]{Assumption}

\numberwithin{equation}{section}

\def\Cov{\mathop{\rm Cov}\nolimits}

\def\d{\mathrm{d}}

\def\<{\langle}
\def\>{\rangle}

\def\b{\beta}
\def\e{\e}

\def\s{\sigma}

\def\vt{\vartheta}

\def\N{{\Bbb N}}  
\def\P{{\Bbb P}}  
\def\E{{\Bbb E}}  

\def\Va{{\Bbb V}}  

\let\cal=\mathcal

\def\OO{{\cal O}}
\def\PP{{\cal P}}

\def\RR{{\cal R}}

\def\ZZ{{\cal Z}}

 \def \b {{\beta}}
 \def \e {{\varepsilon}}
 \def \s {{\sigma}}

 \def \d {{\delta}}

 \def \ba {\begin{array}}
 \def \ea {\end{array}}

 \newcommand{\be}{\begin{equation}}
 \newcommand{\ee}{\end{equation}}

\newcommand{\bea}{\begin{eqnarray}}
 \newcommand{\eea}{\end{eqnarray}}
 
\def\TH(#1){\label{#1}}\def\thv(#1){\ref{#1}}
\def\Eq(#1){\label{#1}}\def\eqv(#1){(\ref{#1})}

 \def \1{\mathbbm{1}}
\def\wt {\widetilde}

\def\wb{\bar}

\def \so {\text{\small{o}}(1)}
\def \po {\text{\small{o}}}

\begin{document}

\title{Stochastic Modelling of T-Cell-Activation}

\author{Hannah Mayer \and Anton Bovier}
 \institute{H. Mayer \at Institut f\"ur Angewandte Mathematik,
Rheinische Friedrich-Wilhelms-Universit\"at, Endenicher Allee 60, 53115 Bonn, Germany\\
\email{hannah.mayer@uni-bonn.de} 
\and
A. Bovier \at
Institut f\"ur Angewandte Mathematik,
Rheinische Friedrich-Wilhelms-Universit\"at, Endenicher Allee 60, 53115 Bonn, Germany\\
\email{bovier@uni-bonn.de}}

\date{Received: date / Accepted: date}
 \maketitle

\begin{abstract}
We investigate a special part of the human immune system, namely the activation of T-Cells, using stochastic tools, especially sharp large deviation results.
T-Cells have to distinguish reliably between foreign and self peptides which are both presented to them by antigen presenting cells. Our work is based on a model studied by Zint, Baake, and
den Hollander in \cite{BZH}, and originally proposed by van den Berg, Rand, and Burroughs in \cite{BRB}. We are able to dispense with some restrictive distribution assumptions that were used previously, i.e. we establish a higher robustness of the model.
A central issue is the analysis of two new perspectives to the scenario (two different quenched systems) in detail. 
This means that we do not only analyse the total probability of a T-Cell activation (the annealed case) but 
also consider the probability of an activation of one certain  T-Cell type  and the probability of a T-Cell activation by a certain antigen presenting cell (the quenched cases). Finally, we see analytically that the probability of T-Cell activation increases with the number of presented foreign peptides in all three cases. 
\keywords{Immune system; T-Cell activation; large deviations; cross-reactivity}
\subclass{60F05, 60F10, 60F17, 60K37, 62P10, 92C37} 
\end{abstract}

\section{Introduction and model setting}

In the present paper we analyze  a stochastic model for T-Cell activation 
that was introduced by van den Berg et al. \cite{BRB,VR}
and later developed and studied by Zint et  al. \cite{BZH}.
While we allow for a  slightly more general model setting, 
the  main new contribution is the analysis  of some different experimental settings that correspond mathematically to various conditional probabilities.
This is explained in detail below.

\subsection{Biological perspective}

Let us first of all say that the model we consider   concerns only  
the mechanism of T-Cell activation. There are many other processes involved in the immune response
 that are not considered here at all.  \emph{T-Cells} have
  the task to recognize foreign antigens against a noisy background of the 
body's own antigens. Any substance which is able to elicit an immune response 
is called \emph{antigen} (from the
term \emph{anti}body \emph{gen}erating). T-Cells do not react with free 
antigens present in the body, but only with antigens presented 
by \emph{antigen presenting cells} (APCs). APCs collect
material in the body, internalize it and split it up in peptides which are 
afterwards presented on the surface of the APC. 

During a so called \emph{immunological synapse}, a bond between a T-Cell and an APC, the T-Cell
scans the presented mixture of peptides using its receptors. 
There is only one receptor type on each T-Cell, but many different peptide 
types on each APC. The task of a T-Cell is to decide on the basis of signals received during an immunological synapse whether foreign antigens are 
present in the body, and to trigger an immune reaction when indicated (in reality, this
 involves a complex interplay with  other parts of the immune system which we do not deal with here). 
This task is made difficult by the following fact: According to Mason \cite{MA} and Arstilla
 et al. \cite{ACB}, there exist about 
$10^{13}$ peptide types  that should be recognized but only $10^7$ different receptor types. 
This implies that a fully 
 specific recognition, i.e. that each receptor type recognizes exactly one 
specific antigen, is impossible. Therefore, a certain degree of 
\emph{cross-reactivity} has to be assumed. The presence of auto-immune diseases, allergies and heavy diseases caused e.g. by viruses and bacteria shows that the immune system faces a very difficult task.

\subsection{The model}

In  \cite{VR} and \cite{BZH} a mathematical model was presented that allows to 
interpret the functioning of  T-Cells as a statistical test problem.  
We briefly describe this model following \cite{BZH}.

\paragraph{Characterization of the APC.}

An APC is  characterized by the types of the presented peptides and the 
number of copies of each peptide type. Thus, each APC can be 
represented by a set of parameters $z_j$  representing 
 the numbers of peptides of type $j$. Peptide types are sometimes 
distinguished as \emph{constituent} (i.e. being present in all cells) 
and \emph{variable} (i.e. being present only in some cells), but this 
distinction  plays no major r\^ole in the present paper (see, however \cite{BL}).
The index 
 $j$ ranges over all peptide types present on this APC. 
We denote the number of foreign peptides present on the APC by $z_f$ and allow,
for simplicity,  for only one type of foreign peptides on one APC. 
Note that the peptides  on the APC are collected by 
a given APC and  thus are a \emph{random} sample   from the total pool of peptides 
present in the body.

\paragraph{Characterization of the T-Cell.}

Each T-Cell possesses one receptor type on its surface. The T-Cell is characterized by the interaction of its receptors with the different peptide types. To each receptor type 
$i$ corresponds a set of association rates, $a_{ij}$, and dissociation rates, 
$r_{ij}$. Here $i$ ranges over all receptor types and $j$ over all peptide
 types. We use the following assumption in line with the approach in \cite{BZH} and \cite{BRB}.

\begin{assumption}\label{Receptors}
There is an abundance of receptors on each T-Cell in the region of interaction with the APC such that each 
released peptide is immediately bound by a receptor again. Thus, association can be assumed as instantaneous and 
 the association rates play no r\^ole.
\end{assumption}
\begin{remark}
We  see later  that we have to work with random stimulation rates to investigate the event of T-Cell 
activation. It will become clear that the qualitative behaviour does not rely on their exact distribution, 
and thus assumption \ref{Receptors} is not too restrictive.
\end{remark}
Under  Assumption \ref{Receptors}, a T-Cell is fully characterized by the set $r_{ij}$, where $i$ refers to the
particular T-Cell. We will see later that we are interested in the duration of peptide-receptor-complexes. 
These duration for  a complex of type $ij$ are denoted by $t_{ij}$ and depend on the dissociation rates.

\paragraph{Activation criteria.}

The interaction of the cells produces a stimulating signal which the T-Cell receives.  This signal results in an activation of the T-Cell, if certain criteria are met. 
The papers  \cite{R,VMC,VL2} and \cite{VL} suggest the following assumption.

\begin{assumption}\label{ActCriteria}\hspace{1cm}\begin{enumerate}
	\item A T-Cell receives a stimulus if a peptide-receptor-complex exists longer than a
	 time $t_*$.
	\item A T-Cell sums up all the stimuli it gets, even those of different receptors.
	\item A T-Cell is activated if the sum of all stimuli exceeds a 
	threshold value $g_{act}$.
\end{enumerate}
\end{assumption}
Note that activation is induced  by stimulation and thus  
the stimulation rates $w_{ij}$ which result from a peptide-receptor-complex of type $ij$ are important. 
The central quantity to analyse  is then the \emph{total stimulation rate} a T-Cell receives. It is denoted by $g_i$ and compounds  all the parameters  mentioned before.

The preliminary description demonstrates that we are concerned with a really high-dimensional problem and model. Merely the number of parameters for the dissociation rates in the model is at least of the order $10^{20}$ ($10^7$ receptor types times $10^{13}$ peptide types).

\paragraph{Stochastic model.}

The large number of parameters in this model makes a specification of all of 
them impossible.  Therefore, van den Berg et al. \cite{BRB} proposed a stochastic model.

A closer look at the system reveals that we must deal  with different 
types and sources of randomness.  This is crucial for the interpretation of 
the probabilities in specific experimental setups.
The dissociation rates are assumed to be random because they are unknown in detail. The presentation profile, i.e. which number of which peptide type is 
present on an APC, is random because it represents a random sample of the 
peptides present in the body. In fact, this quantities are doubly stochastic 
because we have a random pool of peptides in the body and therefrom a random 
sample is presented on the APC. For the sake of simplicity we do not model this 
double effect here. A further reason to assume randomness is that we consider a random encounter of two randomly meeting cells.  
Let us now specify the stochastic version of the model precisely.

Let $(\Omega, \mathcal F, \mathbb P)$ be a probability space on which we define the following random variables. First, 
we  characterize the presentation pattern on the APC.

\begin{enumerate}[(i)]
	\item We set $n\equiv n_c+n_v+1$, where $n_c$ and $n_v$ denote the number of constitutive and variable peptide types on each APC.  
	\item $N_c$ and $N_v$ denote the number of constitutive and variable 
peptide types in the body.
	\item The sample of constitutive and variable peptides is represented 
by positive random variables $Z_j^c$  and $Z_j^v$, each of them 
representing the number of copies of a certain peptide type; they are 
independent and identically distributed (i.i.d.) in the class of constitutive 
and variable peptides, respectively.
	\item $Z_j^c$ and $Z_j^v$ are bounded, and independent of each other. 
	\item $\ZZ$ is the $\sigma$-algebra generated by $Z_j^c$ with 
$j\in \lbrace 1,\dots,N_c \rbrace$ and $Z_j^v$ with $j\in \lbrace 1,\dots,N_v\rbrace$.
	\item We use a short hand notation for the conditional distribution $\P^{\ZZ}(A)\equiv \mathbb P (A|\ZZ)$ and denote the corresponding conditional 
expectation by $\E^{\ZZ}[\cdot]$.  
\end{enumerate}
It is reasonable to assume that $Z_j^c$ and $Z_j^v$ are bounded because the space on each APC is bounded. 

Next, we come to the characterization of the T-Cell.
\begin{enumerate}[(i)]
	\item The total number of the receptor types is $N_1\in \mathbb N$, the one of the peptide types is $N_2$.
	\item The index $i$ denotes the receptor type and $j$ denotes the peptide type.
	\item The dissociation rates of a complex of type $ij$ are positive, i.i.d. random variables $R_{ij}$ with distribution $\PP$ and expectation value $\E [R_{ij}] =	\overline{\tau}$, $	\overline{\tau} \in \mathbb R_+$.
	\item $\RR$ is the $\sigma$-algebra generated by $R_{ij}$ with $i\in\lbrace 1,\dots,N_1\rbrace$ and $j\in \lbrace 	1,\dots,N_2\rbrace$.
	\item The times a certain peptide-receptor-complex  of type $ij$ exists are positive i.i.d. random variables $T_{ij}$ with conditional 
distribution	$\P^{\RR}(T_{ij}\in A)\equiv  \mathbb P (T_{ij}\in A|\RR)$ and corresponding conditional expectation  $\E^{\RR}[T_{ij}]\equiv  \mathbb E [T_{ij}|\RR] = 1/R_{ij}$.
\end{enumerate}
In this setting $T_{ij}$ are random variables in a random environment, in 
other words they are also doubly stochastic.  They display the individual 
duration of a concrete peptide-receptor-complex of type $ij$. Therefore, the 
expected duration of such a bond should be reciprocally proportional to the
 corresponding dissociation rate $R_{ij}$. According to (iv)  
$\E[T_{ij}]=(\bar{\tau})^{-1}$.
The joint distribution of all dissociation rates is given by the product measure 
$\PP^{N_1N_2}$.
The presented sample is from the biological point of view independent of the dissociation rates and the duration of the peptide-receptor-bonds. Thus, $\ZZ$ and $\RR$ are independent.

The parameters of the previous part can be considered as realizations of these random variables.

\paragraph{The stimulation rates.}

According to Assumption \ref{ActCriteria} a single bond between a receptor $i$ and a presented 
peptide of type $j$ during a synapse results in  a stimulation signal if the binding time, $T_{ij}$, exceeds a certain threshold $t_*$. It is assumed that the relevant signal for the T-Cell is the compound average number of stimuli during a synapse, 
\begin{equation}
	\frac 1t\sum\nolimits_{k=1}^{N_j(t)}\mathds 1_{T_{ij}^{k}>t_*},
\end{equation}
where $N_j(t)$ denotes the total number of bindings with a peptide of type $j$  and $T_{ij}^k$ are the respective binding times. Assuming that the number
of bindings is very large, it is reasonable to 
assume that the relevant signal   a T-Cell is receiving from a given peptide type is given by 
\begin{equation}
W_{ij}\equiv \lim_{t\uparrow \infty}
	\frac 1t\sum\nolimits_{k=1}^{N_j(t)}\mathds 1_{T_{ij}^{k}>t_*}
=\frac{\P^{\RR}(T_{ij}>t_*)}{\E^{\RR}[T_{ij}]} \quad \text{a.s.},
\end {equation}
where the last equality follows from elementary renewal theory. We will 
henceforth consider the random variables $W_{ij}$ as the fundamental characteristics of a peptide-receptor interaction. 

 \begin{notations}
 Let the variable $z_f$ denote the number of presented foreign peptides of one particular type. The expected total number of peptides present at one APC is given by $n_M\equiv n_c\mathbb E [Z_1^c]+n_v\mathbb E [Z_1^v]$. The factor $q_n\equiv (n_M-z_f)n_M^{-1}$ used in the following ensures a proportional displacement of the presented self peptides by the foreign peptides\footnote{One may argue that 
scaling the random variables by a common factor 
 is not the best choice to achieve a constant expectation. E.g., if the $Z_j$ 
were assumed to be binomial random variables, it would be more 
reasonable to change the parameter in an appropriate way. However, 
this appears to have only little affect on the results and we keep following 
Zint et al. at this point.}.
 \end{notations}
We drop the index $i$ in favour of a clear notation since it is fixed for one T-Cell. The discussion above motivates Zint et al. \cite{BZH} to define the 
\emph{total stimulation rate} as follows.
\begin{definition}\label{total stim}
With the  notation introduced above, 
 the total stimulation rate a T-Cell receives is given by
\begin{equation}
G_n(z_f)=q_n\left(\sum_{j=1}^{n_c}Z_j^cW_j+\sum_{j=n_c+1}^{n_c+n_v}Z_j^vW_j\right)+z_fW_f.
\end{equation}
\end{definition}
Note that $W_f$ has the same distribution as the other random variables representing stimulation rates, and is independent of these random variables and $\ZZ$.

\paragraph{Interpretation of the probabilities.}

The central quantity for our investigation is the probability of T-Cell activation. If we consider just one certain encounter of an APC and a T-Cell the ``probability'' of T-Cell activation is either $0$ or $1$. But one single experiment does not give much information on the actual situation because it is possible that this T-Cell performs a mistake. So, we and in fact the immune system take another view to the scenario. We are concerned with investigating the following three different cases:

\begin{enumerate}[1.)]
	\item The annealed case: Here we consider the total probability that any T-Cell is activated by 	any APC, $\mathbb P(G_n(z_f)\geq g_{act})$. This probability can be interpreted as the frequency that an activation occurs at all. If $m$ denotes the number of observed experiments, this is the number of meetings of APCs and T-Cells, then 
	\begin{equation}
	\frac 1m \# \text{ \{activated T-Cells\}} \to \P(G_n(z_f)\geq g_{act}).
	\end{equation}
	\item The case quenched with respect to (w.r.t.) the dissociation rates of the T-Cell: In this case we investigate the conditional probability that a certain T-Cell type is activated by any APC, $\P^{\RR}(G_n(z_f)\geq g_{act})$. Here, one type of T-Cell is examining several peptide samples and the results are averaged. The probability represents the frequency of activations of T-Cells of type $i$ during different experiments. If $m$ denotes the number of meetings of T-Cells of type $i$ and different APCs , then 
	\begin{equation}
	\frac 1m \# \text{\{activations of T-Cells of type $i$\}} \to \P^{\RR}(G_n(z_f)\geq g_{act}).
	\end{equation}
	\item The case quenched w.r.t. the environment presented by an APC: The conditional probability that a certain APC activates any type of T-Cell, $\P^{\ZZ}(G_n(z_f)\geq g_{act})$, is analysed here. If $m$ denotes the number of meetings of a certain APC with different T-Cells, then 
	\begin{equation}
	\frac 1m \# \text{\{activations of  T-Cells by the given APC\}}\to \P^{\ZZ}(G_n(z_f)\geq g_{act}).
	\end{equation}
\end{enumerate}
In \cite{BZH} only Case 1.) is considered. In our view the Cases 2.) and 3.) 
are even more relevant.  The probability in Case 2.) is to be interpreted as the frequency of activation of a given T-Cell in repeated 
synapses with different presenter cells. We shall see that these probabilities
depend strongly on the sensitivity of the given T-Cell to the particular presented foreign peptide.
The main question is which information on $z_f$ is deducible from the 
behaviour of the T-Cells. This information emerges only on the basis of 
repeated experiments and the resulting frequency of activations. This leads to the question whether there exists a threshold value $g_{act}$ such that the presence of invaders is distinguishable from the self-background.

\paragraph{Scaling and asymptotics.}  To obtain sensible analytic results
one needs to consider certain numbers to be \emph{large}. Clearly, the total 
number of  peptide types, $n=n_c+n_v+1$, will be assumed the main large 
parameter. The number of foreign peptides has to be seen in comparison to this number, i.e. we understand that $z_f$ \emph{depends} on $n$, and one wants to 
know how large $z_f$ has to be (as a function of $n$) 
for a reliable detection.  The numbers $n_c$ and $n_v$ can in principle also 
be of different magnitude, e.g. one may think that $n_v\sim n$ and $n_c\sim
n^\b$, $\b<1$. For simplicity, we will consider here only the case $n_c\sim n_v\sim n$. 

\section{Results}\label{results}

\subsection{Expectation value and variance of the total stimulation rate}

We compute the expectation value and variance of the total stimulation rate as a function of the number of foreign peptides, $z_f$, in  Cases 1.) , 2.), and 3.). It is easy to calculate these quantities and they allow a first insight into why and in which range a foreign self distinction may work. We assume the existence of the involved first and second moments.

\paragraph{Case 1.)}
According to the following lemma an  increasing number of foreign peptides does not change the expectation value but the variance of the total stimulation rate.

\begin{lemma}\label{varexp}
It holds that $\mathbb E [G_n(0)]=\mathbb E[G_n(z_f)]$,  and that  
\begin{equation}\label{varexp.1001}
\textstyle \mathbb V[G_n(z_f)]-\mathbb V[G_n(0)]=\left( \mathbb V [W_1] +\frac{\mathbb V [G_n(0)]}{n_M^2}\right)z_f\left(z_f-\frac{2n_M\mathbb V[G_n(0)]}{\mathbb V [W_1]n_M^2+ \mathbb [G_n(0)]}\right).
\end{equation} 
\end{lemma}
We can see that the expectation value of $G_n(z_f)$ is independent of $z_f$. The variance depends quadratically on $z_f$, the difference  of the variances as a function of $z_f$ is a parabola with roots
$2n_M\mathbb V[G_n(0)]/(\mathbb V [W_1] n_M^2+\mathbb V [G_n(0)])$ and $ 0$.  
The function decreases first but for $z_f$ large enough it is monotonously increasing and positive. The qualitative behaviour shows up in this general setup, the exact shape depends only on the first and second moments of the involved random variables and not on their exact distribution. In the setting of an increasing variance and a constant expectation value the probability to exceed a threshold value larger than the expectation value increases. Thus, a T-Cell activation may become more likely for a larger value of $z_f$. The increasing variance may allow to have the threshold value $g_{act}$ on a level such that permanent reactions are avoided but an activation becomes more probable. We see here already that a detection can only be  possible if $z_f$ is large enough.

\paragraph{Case 2.)}
Because we consider conditional probabilities and expectations here, these quantities can coincide at most almost surely.  We consider now the difference of the conditional expectations as a function of $z_f$, too. In this case the expectation value and the variance depend on $z_f$.
\begin{lemma}\label{varexp2}
It holds that $\E^{\RR}[G_n(z_f)]-\E^{\RR}[G_n(0)] = z_f(W_f-\E [W_1])$ almost surely,
 and that
\begin{align}\label{varexp2001}
	\Va^{\RR}[G_n(z_f)]-\Va^{\RR}[G_n(0)] =&\left(\Va^{\RR}[W_1]+\frac{\Va^{\RR}[G_n(0)]}{n_M^2}\right)z_f\left(z_f-\frac{2n_M\Va^{\RR}[G_n(0)]}{\Va^{\RR}[W_1]n_M^2+ \Va^{\RR}[G_n(0)]}\right)\nonumber \\
	=& \frac{\Va^{\RR}[G_n(0)]}{n_M^2}z_f\left(z_f-2n_M\right),
\end{align}
where $\Va^{\RR}[X]:=\E^{\RR}[(X-\E^{\RR}[X])^2]$.
\end{lemma}
Note that $\Va^{\RR}[W_1]=0$ because $W_1$ is measurable w.r.t $\RR$. $\E[W_1]$ is with positive probability not equal to $W_f$  because we do not assume the stimulation rates to be distributed according to a Dirac measure. This is reasonable  because we know from the biological background that the stimulation rates can vary. Hence, we have here an effect on the average stimulation rate a T-Cell receives which depends on the stimulation rate associated to the foreign peptide. The conditional expectation of $G_n(z_f)$ increases with $z_f$ if $W_f>\E [W_1]$.  The effect becomes enlarged for an increasing $z_f$. On the other hand, the variance is decreasing for an increasing value of $z_f$ because $z_f$ will never reach $2n_M$ from the biological point of view.  Thus, the situation is really different here.  The parabolas describing the difference of the variances are random here.

\paragraph{Case 3.)}
Again, the conditional expectations can be at most almost surely equal. Here, the expectation value is almost surely independent of $z_f$ but the variance depends again on $z_f$.
\begin{lemma}\label{varexp3}
It holds that $\E^{\ZZ}[G_n(z_f)]-\E^{\ZZ}[G_n(0)] = 0$ almost surely,
and that
\begin{align}
	&\Va^{\ZZ}[G_n(z_f)]-\Va^{\ZZ}[G_n(0)] 
	= &\left(\Va^{\ZZ}[W_1]+\tfrac{\Va^{\ZZ}[G_n(0)]}{n_M^2}\right)z_f\left(z_f-\frac{2n_M \Va^{\ZZ}[G_n(0)]}{\Va^{\ZZ}[W_1]n_M^2+\Va^{\ZZ}[G_n(0)]}\right),
\end{align}
where $\Va^{\ZZ}[X]:=\E^{\ZZ}[(X-\E^{\ZZ}[X])^2]$.
\end{lemma}
In this case the conditional expectation of the total stimulation rate is only almost surely independent of $z_f$. We also have random parabolas for the difference of the variances in this case.  Note that this scenario is very similar to the one in the annealed case; there is an effect on the level of the variance but not on the level of the expectation.\\

\begin{figure}[t]
	\begin{subfigure}{0.48\textwidth} 
 		 \includegraphics[width=\textwidth]{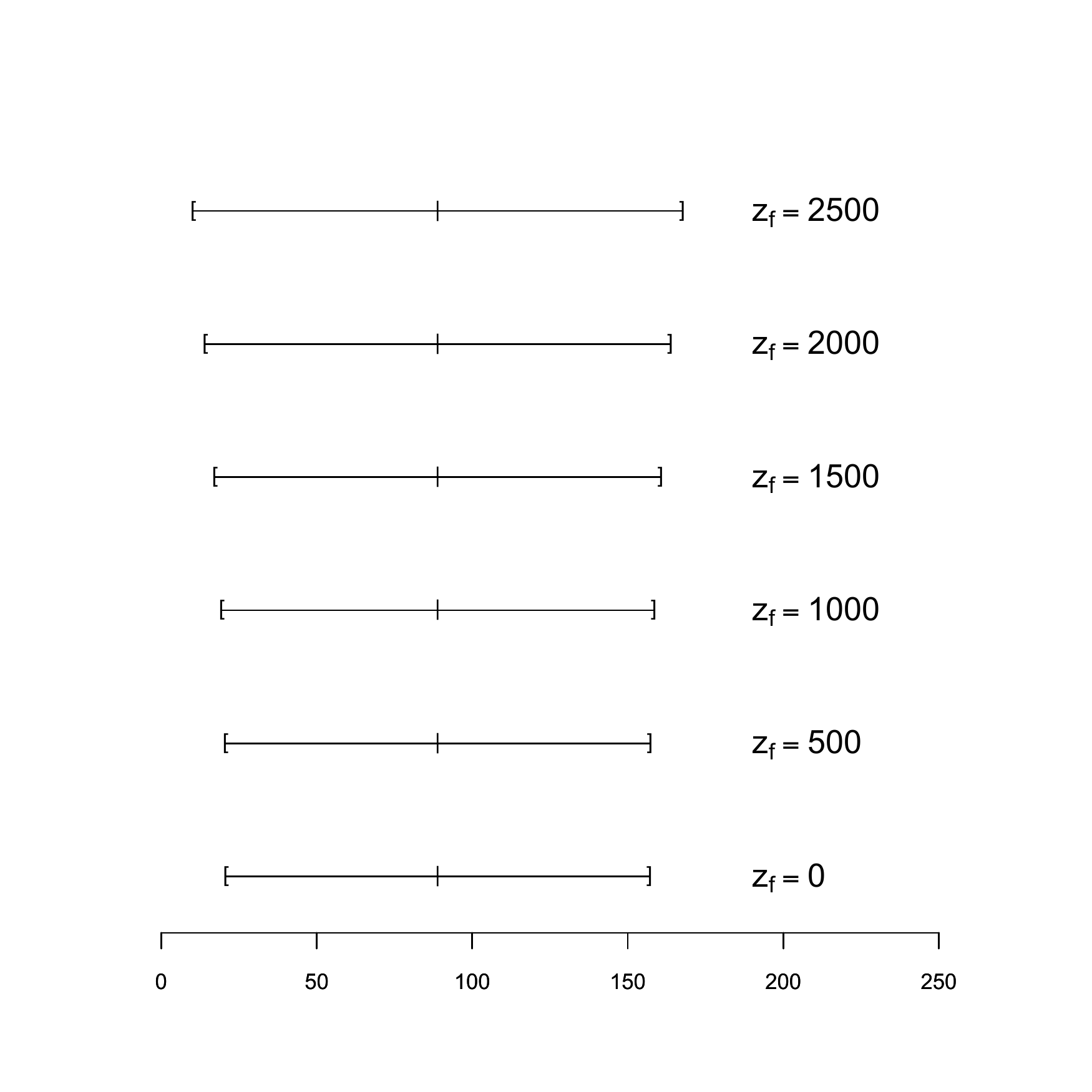}
  		\caption{Case 1.)}
 	 \end{subfigure}
 	 \begin{subfigure}{0.48\textwidth}
		\includegraphics[width=\textwidth]{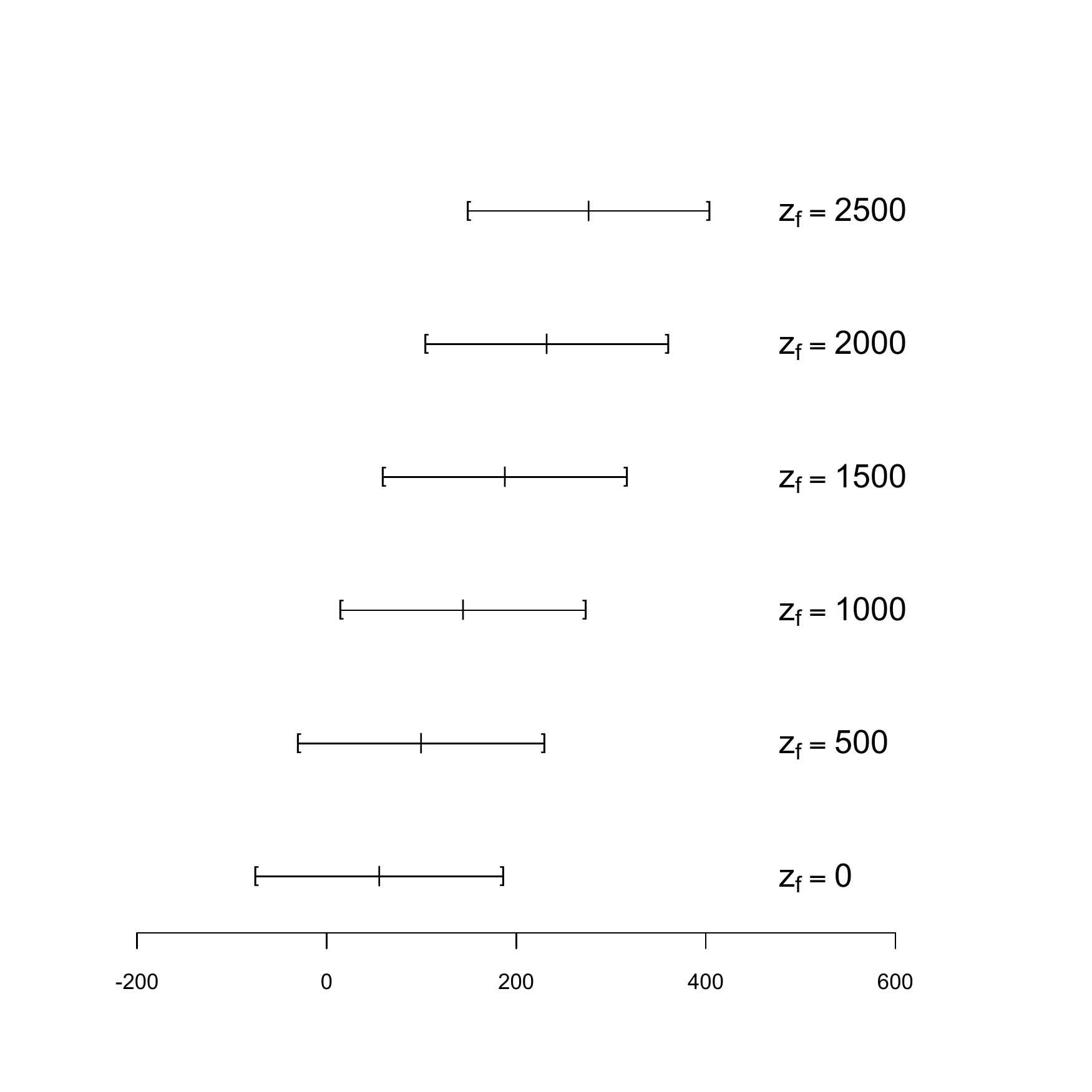}
  		\caption{Case 2.)}
	\end{subfigure} \\
  	\centering 
	\begin{subfigure}{0.48\textwidth}
  		\includegraphics[width=\textwidth]{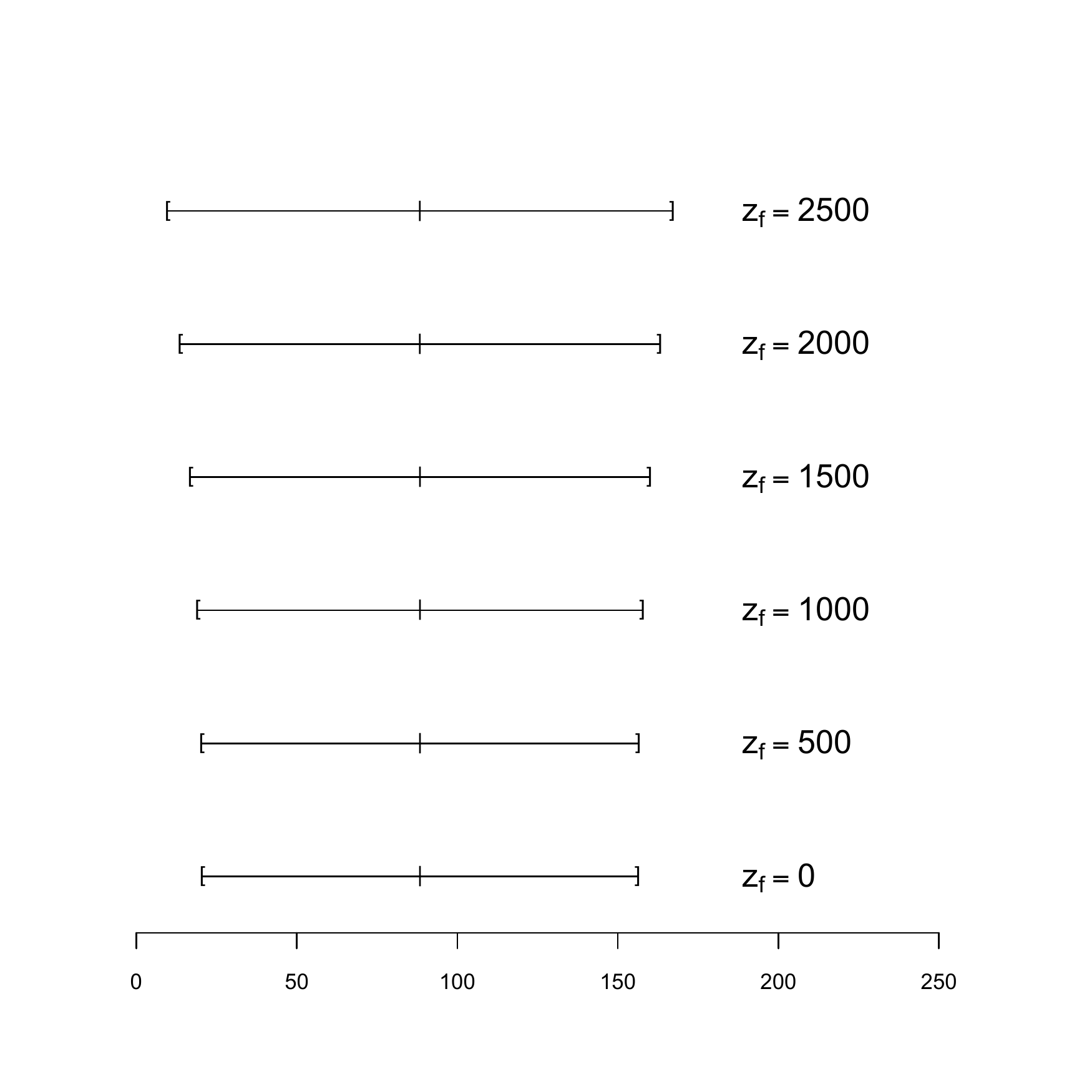}
  		\caption{Case 3.)}
	\end{subfigure}
	\caption{Intervals containing $G_n(z_f)$ with probability and conditional probability 0.99}
	\label{ConfidenceIntervals}
\end{figure}

In Figure \ref{ConfidenceIntervals} intervals are plotted for 
different values of $z_f$,  such that the random variable $G_n(z_f)$  lies in these intervals with probability $0.99$, under the 
tentative assumption that the standardised version of the total stimulation rate,  
$(G_n(z_f)-\E[G_n(z_f)])(\sqrt{\Va[G_n(z_f)]})^{-1}$, is standard normally distributed.  In Case 1.) the intervals enlarge for $z_f$ large enough because the variance of $G_n(z_f)$ increases. In Case 2.) the intervals shrink due to the decreasing variance but they move according to the value of the stimulation rate of the foreign peptide. In the presented case $W_f$ is larger than its expectation $\E[W_f]$. Thus, the intervals move to the right hand side. Case 3.) is very similar to Case 1.) except for a slight difference concerning the expectation value which is just almost surely independent of $z_f$, and thus slightly varying.

\subsection{Large deviations}

 Encounters of T-Cells and APCs happen permanently in the body,  
but only very few of them give rise to an immune reaction.
Therefore, T-Cell activation must be tuned (by suitable choice of the
activation thresholds  $g_{act}$)  such that  
activation and the corresponding immune response are rare events. 
This implies that to compute the activation probabilities, one needs to 
use large deviation techniques (see  for example \cite{FdH} or \cite{DZ}).
  In particular, the computations of means 
and variances from the previous subsection are insufficient.

We are concerned with a family of real-valued random variables $(S_n)_{n\in\N}$ and the probability that $n^{-1}S_n$ exceeds a threshold value $a$, $\P(S_n\geq na)$. A \emph{large deviation} is a deviation from the expectation value of $S_n$ of the order $n$.
If the family of random variables under consideration satisfies a so called 
\emph{large deviation principle} (LDP), the probability for a large deviation event decays exponentially in $n$ with rate $I(a)$. 
The \emph{rate function} $I(a)$ is obtained as the limit of 
the Fenchel-Legendre transformation of the logarithmic moment generating function of $S_n$, to wit $I(a)=\lim_{n\uparrow\infty} I_n(a)$, where
\begin{equation}
 I_n(a)=\sup_{\vt}(a\vt-\Psi_n(\vt)) \equiv a\vt_n-\Psi_n(\vt_n),
\end{equation}
and $\Psi_n(\vt)\equiv \frac 1n\ln\E[\exp(\vt S_n)]$ and $\vt_n$ satisfies 
\begin{equation}\label{theta}
\Psi_n'(\vt_n)=a.
\end{equation}
$\vt_n$ is known as the \emph{tilting parameter} and is used to perform an exponential change of measure in many proofs of theorems in this field. 
The standard theorems of Cram\'er and G\"artner-Ellis then state
\begin{equation}
\P(S_n\geq an)=\exp(-nI(a)(1 + \so)).
\end{equation}
As was pointed out in Zint et al. \cite{BZH}, this approximation 
is not sufficiently precise to calculate the actual probability because of
 the huge and poorly controlled multiplicative error term 
$\exp(n\so)$.
Fortunately,  there are stronger theorems available, 
known as \emph{sharp large deviation results} or \emph{exact asymptotics}. 
 This yields  
approximations  of the form  
\begin{equation}
\P(S_n\geq an)=\frac{\exp(-nI_n(a))}{\vt_n\s_n\sqrt{2\pi n}}(1+\so).
\end{equation}
with the same notation as before and $\s_n^2\equiv \Psi_n''(\vt_n)$. The standard theorem for $S_n$ a sum of i.i.d. random variables is due to 
Bahadur and Rao \cite{BR}. The generalization to independent, but not
 identically distributed random variables, which we need here,
is based on results  of Chaganty and Sethuraman  \cite{CS}. 
We restate these results  in Chapter \ref{proofs}.

These techniques will allow us to achieve our goal, namely  
 to check whether there is a threshold value $g_{act}$ such that the 
probability of activation changes by orders of magnitude with $z_f$. We would like to know if the condition $\mathbb P(G_n(0)\geq g_{act})\ll\mathbb P(G_n(z_f)\geq g_{act})$ can be satisfied for physiologically reasonable values of $z_f$. Therefore, we look at the activation probabilities as a function of $g_{act}$. These functions are often called \emph{activation curves}.

\paragraph{Application to the model of T-Cell activation.}

We look at an artificial sequence of models which is characterized by an increasing number of peptide types, $n\equiv n_c+n_v+1$. 
We assume that there exists $C\in(0,\infty)$ such that $\lim_{n\to\infty} 
n_c/n_v=C$. This implies that there exist $C_1,C_2\in (0,\infty)$ such 
that $\lim_{n\to\infty}n_c/n=C_1$ and $\lim_{n\to\infty}n_v/n=C_2$. 
We use this condition to ensure some convergence properties, 
especially for the rate function. But this is also from the biological point of view a reasonable assumption because the ratio of the numbers of constitutive and variable peptide types is constant. The  sequence of random variables  $S_n$ is given by $G_n(z_f)$ from Definition \ref{total stim}.

In \cite{BZH} certain distributions for all appearing random variables were used to prove the applicability of   Theorem \ref{abwthm} below. 
Afterwards, the approximations of the probabilities were calculated and compared to simulations. Thereby  a separation of the activation curves for different values of $z_f$ was obtained and  a high coincidence of the approximation with the simulation was observed.\\

We prove the applicability of  Theorem \ref{abwthm} in the Cases 1.), 2.),
 and 3.) in Section \ref{proofs} under suitable conditions on the 
distributions and moment generating functions of the involved random 
variables. 
In this section we state just the approximations of the 
probabilities and the involved rate functions. 

\begin{remark}
To obtain the following approximations of the probabilities we assume the existence of the moment generating and the conditional moment generating functions of $Z_1^cW_1$, $Z_1^vW_1$ and $W_1$. This condition is satisfied since we consider bounded random variables. 
\end{remark}

\paragraph{Case 1.)}

Let $M_c(\vt)\equiv \mathbb E[e^{\vt Z_1^cW_1}]$, $M_v(\vt)\equiv \mathbb E[e^{\vt Z_1^vW_1}]$ and $M(\vt)\equiv \mathbb E [e^{\vt W_1}]$ denote the moment generating functions of $Z_1^cW_1, Z_1^vW_1$ and $W_1$. Let $g_{act}(n)\equiv an$ for $a>\E[G_n(z_f)]/n$. The activation probabilities can be approximated by 
\begin{align}
\P(G_n(z_f)\geq g_{act}(n)) = &\frac{\exp(-na\vt_n(a,z_f)+n_c\ln M_c(q_n\vt_n(a,z_f)))}{\vt_n(a,z_f)\sigma_n\sqrt{2\pi n}}\nonumber \\
&\times \exp(n_v\ln M_v(q_n\vt_n(a,z_f))+\ln M(z_f\vt_n(a,z_f)))(1+\so), 
\end{align}
where $\vt_n(a,z_f)$ is chosen such that the argument of the exponential function attains its minimum.
We write $\vt_n(a,z_f)$ and not just $\vt_n$ to visualize the dependence on $a$ and $z_f$. 
We have proven this approximation in analogy to the proof of Zint, Baake,  and  den Hollander in \cite{BZH}. It only requires homogeneous distributions with certain properties in each block. The form of the dependence of the stimulation rates on the dissociation rates is not important. This can be interpreted as a tacit inclusion of competition of the peptides, association rates, loading fluctuations and similar aspects.

The interesting situation is when $z_f$ becomes large, but  remains small  compared to 
$n$. In that case, the infimum will be attained for $\vt_n\sim a$, and 
since the law of $W_1$ is assumed to have bounded support, there will be 
a constant $L$ such that $\frac{d}{d\vt}\ln M(z_f\vt)\sim L z_f$ for $\vt z_f$ large enough. For $z_f \gg \sqrt n$ we may apply this approximation since we know from the applicability of Theorem \ref{abwthm} that $\vt_n\sqrt n\to \infty$. A simple computation then shows that 
\begin{equation}\label{simple.1}
\frac{\P(G_n(z_f)\geq g_{act}(n))}
{\P(G_n(0)\geq g_{act}(n))}\sim \exp\left( \vt_n(a,0) z_f\left(L-a\frac{n}{n_M}\right)
\right),
\end{equation}
where $\vt_n(a,0)$ is the solution of Equation \ref{theta} for $z_f=0$.
This implies that if $a$ is chosen sufficiently small, the activation probability increases exponentially with 
$z_f$, as desired. 
To obtain Equation \ref{simple.1} we used two Taylor approximations: first we expanded $\vt_n(a,z_f)$ in $\vt_n(a,0)$ and after plugging in this term in the rate function we expanded the resulting expression. This way it is possible to recover the probability of activation for the self-background and calculate the ratio of interest as stated by Equation \ref{simple.1}. 
\begin{remark}
The fact that if $a$ is too big, the activation probability drops as $z_f$ increases has a simple intuitive explanation: for very large $a$, the 
contribution of the foreign peptides is limited by the maximal value of $W_f$, 
whereas the reduction of the contribution of the other peptides makes it more
unlikely to achieve an activation by a random fluctuation.
\end{remark}
\paragraph{Case 2.)}
An upper index $\RR$ on any previously defined object should signify the same object conditioned on the $\sigma$-algebra $\RR$. We consider the conditional moment generating functions 
\begin{align}
 M_{\gamma, j}^{\RR}(\vt) = \E^{\RR}[e^{\vt Z_j^{\gamma}W_j}] = \int \exp(\vt Z_j^{\gamma}W_j) dP_{Z_j^{\gamma}}, 
 \gamma \in \lbrace c,v \rbrace ,
\end{align}
where $P_{Z_j^{\gamma}}$ denotes the measure corresponding to $Z_j^{\gamma}, \gamma \in \lbrace c,v \rbrace$. 
Thus, the moment generating functions are random variables themselves. Because the  numbers of copies, $Z_j^c$ and $Z_j^v$, are independent of $\RR$ and the stimulation rates  are measurable w.r.t. $\RR$, these moment generating functions are again
i.i.d. random variables in each block. Due to the measurability of $W_f$ w.r.t. $\RR$ we have $M^{\RR}(\vt)=\exp(\vt z_fW_f)$. The resulting rate function, $I_n^{\RR}(a,z_f)$,  is also \emph{random}. 
We need to apply a law of large numbers in the proof of the approximation of the probabilities in this case. Thus, $\ln (M_{\gamma, j}^{\RR}(\vt))\in L^1(\PP^{N_1N_2})$ 
for each $\vt \geq 0$ and $\gamma \in \lbrace c,v \rbrace$ is an important ingredient for the proof.  Using this fact we can also establish convergence of the rate function according to a strong law of large numbers. But this convergence is not good enough for our purpose because the rate function and thus also the error term arising from the law of large numbers are scaled with a factor $n$ in the large deviation approximation of the probabilities. Therefore, we have to prove a \emph{functional central limit theorem} for a part of the rate function and we have to take into account the term which emerges therefrom. Thereby we obtain a process, $Z_n$, which converges weakly to a Gaussian process. Let 
\begin{equation}\label{g}
g_n(\vt) \equiv \frac{n_c}{n}\E\left[\ln \E^{\RR}\left[e^{\vt Z_1^cW_1}\right]\right] + \frac{n_v}{n} \E\left[\ln \E^{\RR}\left[e^{\vt Z_1^vW_1}\right]\right]
\end{equation}
and $\vt_0^n(a,z_f)$ be defined as the solution of 
\begin{equation}\label{theta0}
a-\frac{z_f}{n}W_f = \frac{d}{d\vt} g_n(q_n\vt). 
\end{equation}
The (random) function $I^n_0(a,z_f)\equiv a\vt_0^n(a,z_f)-g_n(q_n\vt_0^n(a,z_f))$ converges to a function $I_0(a)$. Let 
\begin{align}\label{Zn}
Z_n(\vt)\equiv &\frac{1}{\sqrt n}\left[\sum\limits_{j=1}^{n_c}\left( \ln M_{c,j}^{\RR}(\vt)-\E \left[\ln M_{c,j}^{\RR}(\vt)\right]\right)+\sum\limits_{j=n_c+1}^{n_c+n_v} \left(\ln M_{v,j}^{\RR}(\vt)-\E\left[\ln M_{v,j}^{\RR}(\vt)\right]\right)\right].
\end{align}
With this notation the activation probabilities can be  approximated almost surely according to 
\begin{align}
&\P^{\RR}(G_n(z_f)\geq g_{act}(n))\nonumber \\
= &\frac{\exp\left(-nI^n_0(a,z_f)+\sqrt n Z_n(q_n\vt^n_0(a,z_f))+z_fW_f \vt^n_0(a,z_f)+nR_n\right)}{\s_n\vt_n\sqrt{2\pi n}}(1+\so),
\end{align}
where $R_n\in \OO\left(\frac 1n\right)$.
In Section \ref{proofs} we prove the joint weak convergence of the process $Z_n(q_n\vt_0^n(a,z_f))$  and its derivatives which establishes that the expression for the probabilities is well-behaved.

As in Case 1.), the interesting situation is when $z_f$ becomes large but $z_f/n$ is small. A  computation similar as in Case 1.) then shows that  
\begin{equation}\label{simple.2}
\frac{\P^{\RR}(G_n(z_f)\geq g_{act}(n))}
{\P^{\RR}(G_n(0)\geq g_{act}(n))}\sim \exp\left( \vt_0^n(a,0)z_f\left(W_f-a \frac{n}{n_M}\right)\right).
\end{equation}
This shows that for a given choice of $a$, 
the activation probability increases exponentially with 
$z_f$ only  if $W_f$ is large enough (depending on the choice of 
$a$). 
That is, only the presence of foreign peptides which react strongly with the
particular T-Cell will lead to an increased activation frequency. 
This implies a certain degree of specificity. Although we have seen in Case 1.) that the over all 
probability of T-Cell activation increases with $z_f$, this is not true for any T-Cell type but 
just for those which are equipped with a large enough value of $W_f$.

Note that the  rate function in this case is \emph{random}, that is 
the activation probabilities fluctuate from T-Cell to T-Cell  by a factor
of order $\exp( \sqrt n Z)$, where $Z$ is random. This implies that, for
the modulation of the activiation probabilities due to foreign peptides 
to exceed these random fluctuations significantly, one should have that 
$z_f \gg \sqrt n$. This appears to give limitation on the sensitivity level 
for the recognition of foreign peptides.

\paragraph{Case 3.)}

An upper index $\ZZ$ denotes the objects conditioned on this $\sigma$- algebra. $\wt{\vt}^n(a,z_f)$ is the solution of  
\begin{equation}\label{thetatilde}
a=\frac{d}{d\vt}\wt{g}_n(q_n\vt)+ \frac{d}{d\vt} \frac 1n \ln \E\left[e^{\vt z_fW_f}\right] ,
\end{equation}
where $\wt{g}_n(\vt)\equiv \frac{n_c}{n}\E\left[\ln \E^{\ZZ}\left[e^{\vt Z_1^cW_1}\right]\right] + \frac{n_v}{n} \E\left[\ln \E^{\ZZ}\left[e^{\vt Z_1^vW_1}\right]\right]$. 
$\wt{Z}_n(\vt)$ is defined by 
\begin{align}\label{wtZn}
\wt{Z}_n(\vt)\equiv &\frac{1}{\sqrt n}\left[\sum\limits_{j=1}^{n_c}\left( \ln M_{c,j}^{\ZZ}(\vt)-\E [\ln M_{c,j}^{\ZZ}(\vt)]\right)+\sum\limits_{j=n_c+1}^{n_c+n_v} \left(\ln M_{v,j}^{\ZZ}(\vt)-\E[ \ln M_{v,j}^{\ZZ}(\vt)]\right)\right]
\end{align}
and $\wt{Z}_n(q_n\wt{\vt}^n(a,z_f))$ converges weakly to a Gaussian process. $ \wt I ^n(a,z_f)\equiv a\wt{\vt}^n(a,z_f)-\wt g_n(q_n\wt{\vt}^n(a,z_f))$ converges to a function $\wt I (a)$. 

The notations and the  proof of the result are quite similar to Case 2.) but we can recognize some structural differences in the results here. We have seen these differences between the two conditional scenarios already in the analysis of the variances and the expectation values. Here, the probability of activation can be approximated by 
\begin{align}
&\P^{\ZZ}(G_n(z_f)\geq g_{act}(n))\nonumber \\
= &\frac{\exp{(-n\wt{I}^n(a,z_f)+\sqrt n \wt{Z}_n(q_n\wt{\vt}^n(a,z_f))+\ln \E[e^{\wt{\vt}^n(a,z_f)z_fW_f}]+nR_n)}}{\s_n\vt_n\sqrt{2\pi n}}(1+\so),
\end{align}
where $R_n\in \OO\left(\frac 1n\right)$.
In the rate function appears again a term  which depends on the foreign peptide, namely $\ln \E[\exp{(\wt{\vt}^n(a,z_f)z_fW_f)}]$. But in contrast to Case 2.) this term is deterministic and the only randomness in the rate function lies in the fluctuation term which arises from conditioning. 
Thus, we are again dealing with a random rate function but this function does not vary from experiment 
to experiment by a term which is scaled with $z_f$.

The interesting situation is again when $z_f$ becomes large, but  remains small  compared to 
$n$. As in Case 1.) there will be 
a constant $L$ such that $\frac{d}{d\vt}\ln M(z_f\vt)\sim L z_f$ for $\vt z_f$ large enough. A simple computation then shows that 
\begin{equation}\label{simple.3}
\frac{\P^{\ZZ}(G_n(z_f)\geq g_{act}(n))}
{\P^{\ZZ}(G_n(0)\geq g_{act}(n))}\sim \exp\left( \wt{\vt}^n(a,0) z_f\left(L-a\frac{n}{n_M}\right)
\right),
\end{equation}
where $\wt{\vt}^n(a,0)$ is the solution of Equation \ref{thetatilde} for $z_f=0$. 
 As in Case 2.), there appears  a fluctuation term $\sqrt n Z$. Therefore, we need again $z_f\gg \sqrt n$ such that the impact of this fluctuation term is not too big. This order of $z_f$ ensures again that we may use the approximation $\frac{d}{d\vt}\ln M(z_f\vt)\sim L z_f$.

\begin{remark}
During an infection the body is flooded with the invader. Therefore, a significant ratio of the peptides presented on the APC belongs to the foreign invader and it is reasonable to consider the regime $z_f\gg\sqrt n$. 
A clear indication that a sufficiently high presentation level of the targeted peptides is needed, has been established by 
Landsberg et al. \cite{melanomas2012}  in the context of T-cell therapy of melanomas.
Our work is based on the assumption $z_f\ll n$ although parts of the results in Case 2.) are also valid for $z_f\sim n$. If this is the regime of interest we suggest to use a convolution of the distribution of the foreign stimulation rate and the distribution of the part of $G_n$ belonging to the self background since then the influence of the summand $z_fW_f$ is very large and this summand cannot be treated as the other summands. That is, one should consider $\P(G_n(z_f)\geq an)=\P(q_nG_n(0)\geq an-z_fW_f)=\int \P(n^{-1}G_n(0)\geq q_n^{-1}(a-n^{-1}z_fW_f)|W_f)dP_{W_f}$ and approximate the probability in the integral suitably, depending on the value of $q_n^{-1}(a-n^{-1}z_fW_f)$.
\end{remark}

\section{Precise formulation of  the results and proofs} \label{proofs}

To state the central large deviation result proven by Chaganty and
 Sethuraman in \cite{CS} we introduce some  notation. Let $\lbrace S_n\rbrace_{n \in \mathbb{N}}$  denote a sequence of real-valued random 		variables with moment generating functions 	$\Phi_n (\vartheta) \equiv  \mathbb{E}[\exp(\vartheta S_n)], \vartheta \in \mathbb{R}$ and let $\Psi_n$ be 	defined by 	 $\Psi_n(\vartheta) \equiv  \frac{1}{n} \ln \Phi_n(\vartheta)$.
	\begin{assumption}\label{ass1}
	There exist $\vt^* \in (0,\infty)$ and $\beta < \infty$ such that
	\begin{equation*}
	|\Psi_n(\vartheta)|<\beta, \text{for all } \vartheta \in B_{\vartheta_*}\equiv  \lbrace \vartheta \in \mathbb{C} : |\vartheta| <\vartheta_*\rbrace \text{ and } n\in\mathbb N .
	\end{equation*} 
	\end{assumption}
\begin{notations}
	Let $(a_n)_{n\in\mathbb{N}}$ be a bounded real valued sequence such that  the equation 
	    \begin{equation} \label{Theta}
		  a_n = \Psi_n'(\vartheta)
	    \end{equation}
	has a solution
 	$\vartheta_n \in (0,\vartheta_{**})$ with $\vartheta_{**} \in (0,\vartheta_*) $ for all $n\in\mathbb{N}$.
		$\sigma_n^2 \equiv  \Psi_n''(\vartheta_n)$ is the variance of the tilted version of $n^{-1}S_n$ and $I_n(a_n) \equiv  a_n\vartheta_n - \Psi_n(\vartheta_n)$ is the Fenchel-Legendre transform of $\Psi_n$. We will abusively refer to this as the \emph{rate function}.
\end{notations}
\begin{theorem}[Chaganty and Sethuraman \cite{CS}] \label{abwthm}
	If in the above setting 
  \begin{enumerate}[(i)]
    \item $\lim_{n \to \infty} \vartheta_n \sqrt{n} = \infty $ \label{1}
    \item $\inf_{n\in\mathbb{N}} \sigma_n^2 > 0$ and \label{2}
    \item $\lim_{n \to \infty} \sqrt{n} \sup_{\delta_1 \leq \left|t\right|\leq\delta_2\vt_n} 	\left|\frac{\Phi_n\left(\vt_n+it\right)}{\Phi_n\left(\vt_n\right)}\right|=0 \quad \forall 0<\delta_1<\delta_2<\infty$, \label{3}
  \end{enumerate}
		then
		\begin{equation}
				\mathbb{P}\left(S_n\geq na_n\right) = \frac{e^{-nI_n(a_n)}}{\vt_n\sigma_n\sqrt{2\pi n}}\left(1+\so \right), n\to \infty .
		\end{equation}
\end{theorem}
 We give here the precise conditions that we impose on the distributions of the involved random variables to ensure the applicability of 
Theorem \ref{abwthm} with $S_n=G_n(z_f)$ in the three different cases.

\paragraph{Case 1.)} This case has been considered in Zint et al. \cite{BZH}.
We state their result under slightly more general assumptions.
Below the quantities $\Phi_n,\Psi_n$ and $\vt_n$ are defined as
above with $S_n=G_n(z_f)$. 
\begin{theorem}
 Let $(a_n)_{n\in \mathbb N}$ be defined by $a_n\equiv a$ and $g_{act}(n)=an$ 
 such that $g_{act}(n)>\mathbb E [G_n(z_f)]$ and 
$a<\sup_{\vt\in \mathbb R} \frac{d}{d\vt} \Psi_n(\vt)$ for all $n\in\mathbb N$. Then Theorem \ref{abwthm} is
 applicable provided $z_f/n\downarrow 0$, the distribution functions 
of $Z_1^cW_1$, $Z_1^vW_1$ and $W_1$ are neither lattice valued nor 
concentrated on one point,
 and the corresponding moment generating functions $M_c(\vt),M_v(\vt)$ and $M(\vt)$ are finite for each $\vt\in \mathbb R$. The rate function is
 \begin{equation}
	I_n(a,z_f)=a\vt_n(a,z_f)-\frac{n_c}{n}\ln M_c(q_n\vt_n(a,z_f))-\frac{n_v}{n}\ln M_v(q_n\vt_n(a,z_f))-\frac{1}{n}\ln M(z_f\vt_n(a,zf)).
\end{equation}
\end{theorem}
\begin{proof}
 The moment generating function of the random variable $G_n$ is given by 
 \begin{equation}
 	\Phi_n(\vt) = M_c(q_n\vt)^{n_c}M_v(q_n\vt)^{n_v}M(z_f\vt).
 \end{equation}
It reduces to $\Phi_n(\vt)=M_c(\vt)^{n_c}M_v(\vt)^{n_v}$ if  $z_f=0$.
Assumption \ref{ass1} is satisfied because the following holds:
For each $x\in\mathbb R_+$ and all $\vt<x$
\begin{align}
 \Psi_n(\vt) &\leq \frac{n_c}{n}\ln M_c(q_nx) + \frac{n_v}{n}\ln M_v(q_nx) +\frac 1n\ln M(z_fx) \nonumber \\ 
 &\leq \ln M_c(x)+ \ln M_v(x) +\ln M(z_fx) \equiv \beta(x)
\end{align}
because $\Psi_n(\vt)$ is strictly increasing and $q_n,n_c/n$ and $n_v/n$ are smaller than $1$.
$\vt_n(a,z_f)$ is defined as the (unique) solution of 
\begin{equation}\label{Diff}
 a=\frac{n_c}{n}\left[\frac{d}{d\vt}\ln M_c(q_n\vt)\right] + \frac{n_v}{n}\left[\frac{d}{d\vt}\ln M_v(q_n\vt)\right]+\frac 1n \left[\frac{d}{d\vt}\ln M(z_f\vt)\right].
\end{equation}
This equation results from Equation \ref{Theta} and the choice $a_n\equiv a$. The solution exists since the function $\frac{d}{d\vt}\Psi_n(\vt)$ runs from $\frac 1n \mathbb E[G_n(z_f)]=\frac{d}{d\vt}\Psi_n(\vt)|_{\vt=0}$ to $\sup_{\vt\in\mathbb R} \frac{d}{d\vt}\Psi_n(\vt)$ and 
$a$ lies in between these values. It is unique because $\frac{d}{d\vt}\Psi_n(\vt)$ is strictly increasing.  Because $\frac{n_c}{n}\to C_1$, $\frac{n_v}{n}\to C_2$, $\frac{z_f}{n}\to 0$ and $q_n\to 1$, Equation \ref{Diff} converges. The limit equation is $a=C_1\frac{d}{d\vt}\ln M_c(\vt) + C_2\frac{d}{d\vt}\ln M_v(\vt)$.
Thus, there exists $C\in(0,\infty)$ such that $\lim_{n\to\infty}\vt_n=C$. $C$ is strictly positive because $a>\frac{d}{d\vt}\Psi_n(\vt)|_{\vt=0}$. Consequently, Condition (\ref{1})
of Theorem \ref{abwthm} is satisfied.\\
We define
\begin{equation}
 \sigma_n^2=\left.\left(\frac{n_c}{n}\frac{d^2}{d\vt^2} \ln M_c(q_n\vt) + \frac{n_v}{n}\frac{d^2}{d\vt^2} \ln M_v(q_n\vt)
+ \frac{1}{n}\frac{d^2}{d\vt^2}\ln M(z_f\vt)\right)\right|_{\vt=\vt_n(a,z_f)}.
\end{equation}
This equation converges as the previous one and the second derivatives of $\ln M_{\gamma}(q_n\vt)$, $\gamma\in\lbrace c,v\rbrace$ are positive due to the
 strict convexity of these functions. Thus, Condition (\ref{2}) of Theorem \ref{abwthm} is satisfied, too.\\
We define
\begin{align}
\nu_{\gamma}^n(t) = \frac{M_{\gamma}(q_n(\vt_n(a,z_f) +it))}{M_{\gamma}(q_n\vt_n(a,z_f))}, \gamma \in \lbrace c,v\rbrace  \quad \text{ and }\quad \nu^n(t) = \frac{M(z_f(\vt_n(a,z_f)+it))}{M(z_f\vt_n(a,z_f))}.
\end{align}
These are the characteristic functions of the tilted random variables. The distribution functions corresponding  to these characteristic functions are also neither lattice valued nor concentrated on one point. Because $\vt_n(a,z_f)\to C$ and $q_n\to 1$, there exist $\e>0$ and $n_0<\infty$  for each $t\neq 0$
 such that for all $n\geq n_0$
\begin{equation*}
 |\nu^n_{\gamma}(t)|\leq 1-\e, \gamma \in\lbrace c,v\rbrace \text{ and } |\nu^n(t)|\leq 1-\e.
\end{equation*}
We obtain
\begin{align}\label{cf}
\left|\frac{\Phi_n(\vt_n(a,z_f)+it)}{\Phi_n(\vt_n(a,z_f))}\right| &=\left|\frac{M_c(q_n(\vt_n(a,z_f)+it))^{n_c}M_v(q_n(\vt_n(a,z_f)+it))^{n_v}M(z_f(\vt_n(a,z_f)+it))}{M_c(q_n\vt_n(a,z_f))^{n_c}M_v(q_n\vt_n(a,z_f))^{n_v}M(z_f\vt_n(a,z_f))}\right| \nonumber \\
&=\left|(\nu_c^n(t))^{n_c}(\nu_v^n(t))^{n_v}\nu^n(t)\right|\nonumber \\
&\leq (1-\e)^n = \po\left(\frac{1}{\sqrt{n}}\right), \quad n\to\infty.
\end{align}
It remains to consider the supremum over the values of $t$ in Condition (iii).
 Since $\vt_n(a,z_f)$ converges, the supremum is taken over a compact set. 
Thus, this function attains a maximum on this interval and this can be bounded according to Equation \ref{cf}. Therefore, Condition (\ref{3}) 
of Theorem \ref{abwthm} is satisfied and Theorem \ref{abwthm} is applicable.
\qed\end{proof}

\begin{remark} The case $z_f\sim n$ requires a special treatment. 
This is, however, best relegated to the following Case 2.).
\end{remark}

\paragraph{Case 2.)}
We denote by $\E_{\PP^{N_1N_2}}[\cdot]$ the expectation w.r.t. the measure $\PP^{N_1N_2}$ which is the joint distribution of all the dissociation rates. We would like to show that the conditions of Theorem \ref{abwthm} are almost surely satisfied. As already mentioned it is important that $\ln(M_{\gamma, j}^{\RR}(\vt)) \in L^1(\PP^{N_1N_2})$, where $\gamma\in\{ c,v\}$. Under the assumption $M_{\gamma}(\vt)<\infty$ for each $\vt \in (0,\vt_{**})$ we have that $\E^{\RR}[\exp(\vt Z_1^{\gamma}W_1)]\in L^1(\PP^{N_1N_2})$. Combined with 
\begin{equation}
0\leq \ln (\E^{\RR}[\exp(\vt Z_1^{\gamma}W_1)]) < \E^{\RR}[\exp(\vt Z_1^{\gamma}W_1)], \text{ for } \vt\geq 0,
\end{equation}
this yields
 \begin{equation}
 \ln(M_{\gamma, j}^{\RR}(\vt)) = \ln (\E^{\RR}[\exp(\vt Z_1^{\gamma}W_1)])\in L^1(\PP^{N_1N_2}),
 \end{equation}
where $\gamma\in\{ c,v\}$. Below we denote by $\Psi_n^\RR$, $\vt_n^\RR(a,z_f)\equiv \vt_n(a,z_f)$ the analogues of
the quantities $\Psi_n$, $\vt_n(a,z_f)$ under the conditional 
 expectations 
$\E^\RR$. For notational simplicity we drop that superscript on $\vt_n(a,z_f)$, but it is important to keep in mind that this is now a random variable, too. 

\begin{theorem}
 Let $(a_n)_{n\in\mathbb N}$ be defined by $a_n \equiv a$ and $g_{act}(n)=an$  such that $g_{act}(n) >\E^{\RR}[G_n(z_f)]$ and $a < \sup_{\vt \in \mathbb R} \frac{d}{d\vt} \Psi_n^{\RR}(\vt)$ for all $n\in \mathbb N$. Then Theorem \ref{abwthm} is almost surely applicable  if the distribution functions of the stimulation rates are neither lattice valued nor concentrated on one point and the moment generating functions $M_{c,j}^{\RR}(\vt)$, $M_{v,j}^{\RR}(\vt)$ and $M^{\RR}(\vt)$ as well as $M_c(\vt)$, $M_v(\vt)$ and $M(\vt)$ are finite for each $\vt\in\mathbb R$. Then the rate function is
 \begin{align}
 	I_n^{\RR}(a,z_f) =   &\left(a- \frac{z_f}{n}W_f\right)\vt_n(a,z_f) \nonumber \\
	&-\frac 1n \left(\sum_{j=1}^{n_c}\ln M_{c,j}^{\RR}(q_n\vt_n(a,z_f)) + \sum_{j=n_c+1}^{n_c+n_v} \ln M_{v,j}^{\RR}(q_n\vt_n(a,z_f))\right).
 \end{align}
\end{theorem}
\begin{proof}
	Due to the monotonicity of the logarithmic moment generating function for each realization of the stimulation rates and the boundedness of all involved random variables we can  find again $\beta(x)$ such that $\Psi_n^{\RR}(\vt)<\beta(x)$ for all $\vt < x$. Thus, Assumption \ref{ass1} is satisfied.  
Recall that $\vt_n(a,z_f)$ is defined as the solution of 
the equations 
\begin{equation}
a=\sum_{j=1}^{n_c}\frac d{d\vt}\ln M_c^{\RR}(q_n\vt) +  \sum_{j=n_c+1}^{n_c+n_v}\frac d{d\vt}\ln M_v^{\RR}(q_n\vt)+\frac{z_f}{n}W_f.
\end{equation}
The solution $\vt_n(a,z_f)$ exists due to the choice of $a$ and is unique due to the strict convexity of $\Psi_n^{\RR}$.

For $\vt \geq 0$ we have, by the law of large numbers,  
	\begin{align} \label{schoen.1}
		&\lim_{n\to\infty} \frac 1n\left( \sum_{j=1}^{n_c}\ln M_{c,j}^{\RR}(q_n\vt) + \sum_{j=n_c+1}^{n_c+n_v} \ln M_{v,j}^{\RR}(q_n\vt)\right) \nonumber \\
		=& C_1\E_{\PP^{N_1N_2}}\left[\ln M_{c,1}^{\RR}(\vt)\right] + C_2\E_{\PP^{N_1N_2}}\left[\ln M_{v,1}^{\RR}(\vt)\right], \quad \PP^{N_1N_2}-a.s.
	\end{align}
Since the derivatives of the summands satisfy the bounds
	\begin{align}
		0 & \leq \frac{d}{d\vt} \ln \E^{\RR}[\exp(\vt q_n Z_1^{\gamma}W_1)] = \frac {q_nW_1 \E^{\RR}[Z_1^{\gamma}\exp(\vt q_n Z_1^{\gamma} W_1)]}{\E^{\RR}[\exp(\vt q_n Z_1^{\gamma} W_1)]} \nonumber \\
		& \leq \frac{q_nW_1Z_1^{\gamma,\max} \E^{\RR}[\exp(\vt q_n Z_1^{\gamma} W_1)]}{\E^{\RR}[\exp(\vt q_n Z_1^{\gamma} W_1)]} = q_n W_1Z_1^{\gamma,\max}, 
	\end{align}
where $Z_1^{\gamma,\max}$ denotes the maximal value of $Z_1^{\gamma}$, $\gamma \in \lbrace c, v \rbrace$. Therefore 
they are integrable and hence the limit of the derivatives on the left-hand side of Equation \eqref{schoen.1} exists and is equal to the 
derivative of the right-hand side.  
If either $\lim_{n\to \infty}{z_f/n}=0$, or $\lim_{n\to\infty}{z_f/n}=C>0$, the equations 
determining $\vt_n(a,z_f)$ 
converge almost surely and therefore so does the solution $\vt_n(a,z_f)$. 
 Thus, Condition (\ref{1}) of Theorem \ref{abwthm} is again satisfied.
 
We have $(\frac{d^2}{d\vt ^2} \Psi_n^{\RR} (\vt))|_{\vt=\vt_n(a,z_f)} >0$ for each $n$ due to the strict convexity of $\Psi_n^{\RR}$. So, it remains to check whether this holds true in the limit $n\to \infty$. $\lim_{n\to\infty} \frac{d^2}{d \vt ^2}\Psi_n^{\RR}(\vt)$ exists because the summands of the derivative are again bounded and therefore integrable, since
\begin{align}
	0 &\leq \frac{d^2}{d \vt ^2} \ln \E^{\RR}[ \exp(\vt Z_1^{\gamma} W_1)]  \leq \frac{ q_n^2 W_1^2\E^{\RR}[(Z_1^{\gamma})^2e^{\vt q_n Z_1^{\gamma} W_1}]\E^{\RR}[e^{\vt q_nZ_1^{\gamma} W_1}]-\E^{\RR}[q_nW_1Z_1^{\gamma}e^{\vt q_nZ_1^{\gamma}W_1}]^2}{\E^{\RR}[e^{\vt q_n Z_1^{\gamma}W_1}]^2} \nonumber \\
	&\leq \frac{q_n^2 W_1^2(Z_1^{\gamma, \max})^2 \E^{\RR}[e^{\vt q_nZ_1^{\gamma}W_1}]^2}{\E^{\RR}[e^{\vt q_n Z_1^{\gamma}W_1}]^2} = q_n^2W_1^2(Z_1^{\gamma,\max})^2 .
\end{align}
Thus, it is again allowed to interchange limit and derivative. Moreover, 
we have that 
\begin{equation}
	\frac{d^2}{d\vt ^2}\left(C_1 \E_{\PP^{N_1N_2}}[\ln M_{c,1}^{\RR}(\vt)]\right) = C_1\E_{\PP^{N_1N_2}}\left[\frac{d^2}{d\vt ^2}\ln M_{c,1}^{\RR}(\vt)\right] > 0
\end{equation}
because $\ln M_{c,1}^{\RR}(\vt)$ is strictly convex. Thus, this summand is positive and, analogously, so is the second one. Therefore, Condition (\ref{2}) of Theorem \ref{abwthm} is satisfied. 

Next we  check Condition (\ref{3}) on the characteristic function. We have to take into account  that $Z_j^c$ and $Z_j^v$ should be lattice valued random 
variables because they represent numbers of peptides. 
The characteristic function is given by
\begin{align}
	 \left| \frac{\Phi_n^{\RR}(\vt+ it)}{\Phi_n^{\RR}(\vt)}\right|	= \left|\frac{\prod_{j=1}^{n_c}M_{c,j}^{\RR}(q_n(\vt+it))\prod_{j=n_c+1}^{n_c+n_v}M_{v,j}^{\RR}(q_n(\vt+it)) M^{\RR}(z_f(\vt +it))}{\prod_{j=1}^{n_c}M_{c,j}^{\RR}(q_n\vt)\prod_{j=n_c+1}^{n_c+n_v}M_{v,j}^{\RR}(q_n\vt)M^{\RR}(z_f\vt)}\right| .\label{momgenfct}
\end{align}	
We can rewrite \ref{momgenfct} as
\begin{align}	
	 & \textstyle \left(\exp\left(\frac {1}{n_c}\sum_{j=1}^{n_c} \ln \Big|\frac{M_{c,j}^{\RR}(q_n(\vt+it))}{M_{c,j}^{\RR}(q_n\vt)}\Big|\right)\right)^{n_c}\left(\exp\left(\tfrac {1}{n_v}\sum_{j=1}^{n_v} \ln \Big|\frac{M_{v,j}^{\RR}(q_n(\vt+it))}{M_{v,j}^{\RR}(q_n\vt)}\Big|\right)\right)^{n_v}\Big|\frac{M^{\RR}(z_f(\vt+it))}{M^{\RR}(z_f\vt)}\Big| \nonumber \\
	= & \exp \Biggl(n_c\left(\E_{\PP^{N_1N_2}}\left[\ln \Big|\tfrac{M_{c,1}^{\RR}(q_n(\vt +it))}{M_{c,1}^{\RR}(q_n\vt)}\Big|\right]+\so\right) \nonumber \\
	&+n_v\left(\E_{\PP^{N_1N_2}}\left[\ln \Big|\tfrac{M_{v,1}^{\RR}(q_n(\vt +it))}{M_{v,1}^{\RR}(q_n\vt)}\Big|\right]+\so\right)
	+\ln\Big|\tfrac{M^{\RR}(z_f(\vt+it))}{M^{\RR}(z_f\vt)}\Big|\Biggr) . 
\end{align}	
This expression can be bounded from above by	
\begin{align}
		&\exp\Biggl(n_c\left(\ln(1-\e)\PP^{N_1N_2}\left(\Big|\tfrac{M_{c,1}^{\RR}(q_n(\vt +it))}{M_{c,1}^{\RR}(q_n\vt)}\Big|\leq 1-\e\right)+\tilde{\e}\right)\nonumber \\
	&\qquad +n_v\left(\ln(1-\e)\PP^{N_1N_2}\left(\Big|\tfrac{M_{v,1}^{\RR}(q_n(\vt +it))}{M_{v,1}^{\RR}(q_n\vt)}\Big|\leq 1-\e\right)+\tilde{\e}\right)\Biggr).	\label{schoen.2}
\end{align}
 For given $\e>0$, the  probabilities 
\begin{equation}
\PP^{N_1N_2}\left(\left|\tfrac{M_{c,1}^{\RR}(q_n(\vt +it))}{M_{c,1}^{\RR}(q_n\vt)}\right|\leq 1-\e\right)\quad \text{ and } \quad \PP^{N_1N_2}\left(\Big|\tfrac{M_{v,1}^{\RR}(q_n(\vt +it))}{M_{v,1}^{\RR}(q_n\vt)}\Big|\leq 1-\e\right)
\end{equation}
 are strictly positive, uniformly in $n$ for $n$ large,  
due to the assumptions on the distribution of the stimulation rates.
Therefore, there exists $\d>0$, such that for all $n$ large enough, 
\eqref{schoen.2} is bounded from above by
\begin{equation}
\exp((n-1)(\tilde\e-\delta)).
\end{equation}
Since $\tilde \e$ can be made arbitrarily small if $n$ is large enough,
 $\delta -\tilde\e>0$ for such $n$, and so this expression tends to zero with 
$n$ exponentially fast. 
It is again a crucial point that $\vt_n(a,z_f)$ converges such that the supremum is
 taken over a compact set and Condition (\ref{3}) is satisfied.
\qed\end{proof}

\begin{remark}
Note that it  suffices in order  to check these conditions to assume $C_1+C_2 >0$. It is not necessary that both constants are strictly positive. 
\end{remark}

\paragraph{Investigation of the rate function}
We are concerned with the behaviour of the large deviation rate function and prove a functional central limit theorem with which we can characterize this.  $g_n(q_n\vt)$  defined by Equation \ref{g} converges to
\begin{align}
&C_1\E\left[\ln \E^{\RR}\left[e^{\vt Z_1^cW_1}\right]\right] + C_2\E\left[\ln \E^{\RR}\left[e^{\vt Z_1^vW_1}\right]\right] \equiv g(\vt).
\end{align}
In the rate function appears the process $Z_n(q_n\vt^n_0(a,z_f))$ which is defined by Equation \ref{Zn}. The following theorem states our result.
We use the short hand notation  $M_{c,1}^{\RR,a}\equiv \ln M_{c,1}^{\RR}(\vt_0(a))$, where $\vt_0(a)$ denotes the limit of $\vt^n_0(a,z_f)$, the solution of Equation \ref{theta0}.

\begin{theorem}\label{ratefctR}
If  there exists a constant $C$ such that $g_n''(q_n\vt^n_0(a,z_f))>C>0$, the rate function is given by 
\begin{align}
I_n^{\RR}(a,z_f) = &I^n_0(a,z_f) -\frac{1}{\sqrt n}Z_n(q_n\vt^n_0(a,z_f))-\frac{z_f}{n}\vt^n_0(a,z_f)W_f +R_n,
\end{align}
where $Z_n(q_n\vt_0^n(a,z_f))$
converges weakly to the Gaussian process $Z_{a} + \overline{Z}_{a}$ and $R_n\in \OO\left( \frac 1n \right)$.  $Z_{a}$ and $\overline{Z}_{a}$ are both Gaussian processes with expectation functions $\E[Z_{a}]=0=\E[\overline{Z}_{a}]$  and covariance functions
\begin{align}
\Cov (Z_{a},Z_{a '}) = C_1&\left( \E\left[M_{c,1}^{\RR,a}M_{c,1}^{\RR,a'}\right] -\E\left[M_{c,1}^{\RR,a}\right]\E\left[ M_{c,1}^{\RR,a'}\right]\right)
\end{align}
and
\begin{align}
\Cov (\overline{Z}_{a},\overline{Z}_{a '}) = C_2&\left( \E\left[M_{v,1}^{\RR,a}M_{v,1}^{\RR,a'}\right]
-\E\left[M_{v,1}^{\RR,a}\right]\E\left[M_{v,1}^{\RR,a'}\right]\right).
\end{align}
\end{theorem}
\begin{remark}
The remainder term is given by 
\begin{equation}
R_n = \frac{ (Z_n'(q_n\vt_0^n(a,z_f)))^2}{2n(g_n''(q_n\vt_0^n(a,z_f))+\frac 1{\sqrt n}Z_n''(q_n\vt_0^n(a,z_f)))}+\po\left(\frac 1n\right),
\end{equation}
where the appearing process scaled with $n$ converges weakly. Since we consider the regime $z_f\gg \sqrt n$ the term $\frac{z_f}{n}\vt^n_0(a,z_f)W_f$ is of a higher order than the remainder. 
\end{remark}

As we already mentioned in Section \ref{results} we need this approximation of the rate function on the level of the central limit theorem due to the scaling with the factor $n$ in the expression for the probabilities. 
In order to prove this result we show weak convergence of the involved random processes and derive then an expression for the rate function. To establish the weak convergence of $Z_n(q_n\vt^n_0(a,z_f)), Z_n'(q_n\vt^n_0(a,z_f))$ and $Z_n''(q_n\vt^n_0(a,z_f))$ as well as their joint weak convergence as processes on the Wiener Space with parameter $a$ we show convergence of their finite dimensional distributions and tightness. To prove tightness we use the Kolmogorov-Chentsov criterion from \cite{KAL}. Formulated to our scenario we have to check the conditions
\begin{enumerate}
\item $Z_n(q_n\vt_0^n(a,z_f))$ converges in finite dimensional distribution. \label{convfdd}
\item The family of initial distributions, $Z_n(q_n(\vt_0^n(\e,z_f)))$, is tight. \label{starttight}
\item There exists $C>0$ independent of $a$ and $n$ such that \label{tightness}
		\begin{equation}
		\E\left[\left( Z_n(q_n\vt_0^n(a+h,z_f)-Z_n(q_n\vt_0^n(a,z_f)\right)^2\right]\leq C|h|^2.
		\end{equation} 
\end{enumerate}
Note that Condition \ref{tightness} is fulfilled if
\begin{equation}\label{versiontightness}
\E\left[\left(Z_n'(q_n\vt_0^n(a,z_f))\right)^2\right]\leq C.
\end{equation}
The same criteria with $Z_n(q_n\vt_0^n(a,z_f))$ suitably replaced by the process under consideration can be used to prove the convergence of these processes.

We can handle the constitutive and the variable part separately. It suffices to check the conditions for the constitutive part because the sum in the variable part is built analogously.
The following results are taken from \cite{JS}. We need this central limit theorem for triangular arrays to check Condition \ref{convfdd}.
\begin{definition}
A \emph{row-wise independent $d$-dimensional triangular array scheme} is a sequence $(K^n)$ of elements of $\overline{\mathbb N}^*=\mathbb N \setminus \lbrace 0 \rbrace \cup \infty$ and a sequence of probability spaces $(\Omega^n,\mathcal F ^n, P^n)$  each of one being equipped with an independent sequence $(\chi _k ^n)_{1\leq k \leq K^n}$ of $\mathbb R^d$-valued random variables.
\end{definition}
We restrict the scenario to row-wise independent schemes which satisfy
\begin{equation}\label{A}
	\sum_{1\leq k \leq K^n} \left|\E\left[h\left(\chi_k^n\right)\right]\right|<\infty \text{ and } \sum_{1\leq k\leq K^n}\E\left[\left|\chi_k^n\right|^2\wedge 1\right]<\infty
\end{equation}
for each $n$, where $h$ is a given truncation function. This condition does not depend on $h\in \mathcal C _t^d \equiv \lbrace h:\mathbb R^d\to\mathbb R^d \text{ bounded, compact support, }h(x)=x \text{ in a neighbourhood of } 0 \rbrace$.

\begin{definition}
	A row-wise independent array $(\chi_k^n)$ satisfies the \emph{Lindeberg condition} if for all
	 $\e >0$ we have 
	\begin{equation}\label{Lind}
		\lim_{n\to\infty}\sum_{1\leq k\leq K^n} \E\left[|\chi_k^n|^2\mathds 1_{\lbrace |\chi_k^n|>\e \rbrace}\right] =0.
	\end{equation}
\end{definition}
Of course, this implies $\sum_k \E[|\chi_k^n|^2]<\infty$, provided Condition \ref{A} is satisfied.
\begin{theorem}\label{Lindeberg-Feller}
 We suppose that the $d$-dimensional row-wise independent array satisfies Condition \ref{A} and the Lindeberg condition, and let $\xi^n=\sum_{1\leq k \leq  K^n} \chi_k^n$. Then\\
 a) If $\mathcal L(\xi^n)\to \mu$, then $\mu$ is a Gaussian measure on $\mathbb R^d$;\\
 b) in order that $\mathcal L(\xi^n)\to \mathcal N(b,c)$, the Gaussian measure with mean $b$ and covariance matrix $c$, it is necessary and sufficient that the following two conditions hold:\\
 $[\beta]$ $\sum _{1\leq k\leq K^n} \E\left[\chi_k^n\right]\to b$ \\
 $[\gamma]$  $\sum_{1\leq k\leq K^n} \E\left[\chi_k^{n,j}\chi_k^{n,l}\right]\to c^{jl}$,\\
 where $\chi_k^{n,l}$ denotes the $l$-th component of $\chi_k^n$.
\end{theorem}
Using this theorem we can prove the following lemmata which we need to prove Theorem \ref{ratefctR}.
\begin{lemma}\label{R}
$Z_n(q_n\vt^n_0(a,z_f))$ as a process on the Wiener Space with parameter $a$ converges weakly to a Gaussian process if there exists a constant $C$ such that $g_n''(q_n\vt^n_0(a,z_f))>C>0$.
\end{lemma}
In order to simplify the notation we define 
\begin{equation}
Y_{a,j}^n\equiv\ln \E^{\RR}\left[e^{q_n \vt^n_0(a,z_f)Z_j^cW_j}\right]- \E\left[\ln \E^{\RR}\left[e^{q_n \vt^n_0(a,z_f)Z_j^cW_j}\right]\right].
\end{equation}
  The constitutive part of the process $Z_n(q_n\vt^n_0(a,z_f))$ is given by $Z_{n,c}(q_n\vt^n_0(a,z_f))\equiv \frac{1}{\sqrt n}\sum_{j=1}^{n_c}Y_{a,j}^n$. To prove this lemma we have to check Conditions \ref{convfdd}, \ref{starttight}, and \ref{tightness}.
First we investigate the finite dimensional distributions of $Z_{n,c}$ in the following Lemma \ref{fddR}. Therefore, let  $0<a_1<\dots<a_m<\infty$, $a_i \in \mathbb R, m\in \mathbb N$.
We are interested in the limiting behaviour of $\xi^n \equiv \sum_{1\leq j \leq K^n} \chi_j^n$ with $\chi_j^n \equiv  \frac {1}{\sqrt n}(Y_{a_1,j}^n,\dots,Y_{a_m,j}^n)
$ and $j \in \lbrace 1,\dots,n_c \rbrace$. 
\begin{lemma}\label{fddR}
	Under the assumptions of Lemma \ref{R}, $\xi^n\equiv \sum_{j=1}^{n_c} \chi_j^n$ converges weakly to a Gaussian vector with expectation $0$ and covariance matrix defined by 
	\begin{equation}
	C^{jl}=C_1\left(\mathbb E\left[M_{c,1}^{\RR,a_j} M_{c,1}^{\RR,a_l}\right]-\mathbb E\left[M_{c,1}^{\RR,a_j}\right]\mathbb E\left[M_{c,1}^{\RR,a_l}\right]\right).
	\end{equation}
\end{lemma}
\begin{proof}
We show that Theorem \ref{Lindeberg-Feller} is applicable in this case. We have
\begin{equation}\label{boundR}
	|\chi_j^n|^2=\sum_{i=1}^m \left(\frac{1}{\sqrt n}\ln M_{c,j}^{\RR}(q_n\vt^n_0(a_i,z_f))-\mathbb E\left[\frac{1}{\sqrt n}	\ln M_{c,j}^{\RR}(q_n\vt^n_0(a_i,z_f))\right]\right)^2 \leq \frac {4m}n K^2
\end{equation}
where $K$ is the global constant bounding each $M_{c,j}^{\RR}$ for $\vt \in (0,\vt_{**})$, independent of $j$. We have to check that 
the Lindeberg condition \ref{Lind} is satisfied. Since $|\chi_j^n|\leq 2\sqrt{\frac{m}{n}}K$  there exists for each $\e >0$ $n_0\in \mathbb N$ such that $|\chi_j^n|<\e$ for all $n\geq n_0$. Therefore, each summand is $0$ for $n\geq n_0$ and thus also the sum and the 
limit vanish. Part $1$ of Condition \ref{A} is satisfied because we consider centered random variables. Part $2$  holds true due to
\begin{align}
\sum_{j=1}^{n_c}\E\left[|\chi_j^n|^2\right]\leq \max \sum_{j=1}^{n_c}|\chi_j^n|^2\leq n_c\frac{4 mK^2}{n}\leq 4mK^2<\infty
\end{align}
according to \ref{boundR}. There can only appear finitely many summands which are equal to $1$.  Condition $[\beta]$ of  Theorem \ref{Lindeberg-Feller} is satisfied because each $\chi_j^n$ has expectation $0$ due to the construction.
Condition $[\gamma]$ is satisfied since 
\begin{align}
 \sum_{j=1}^{n_c} \mathbb E \left[\chi_j^{n,k}\chi_j^{n,l}\right] 
= & \frac{n_c}{n} \Biggl(\mathbb E \left[\ln M_{c,1}^{\RR}(q_n\vt^n_0(a_k,z_f))\ln M_{c,1}^{\RR}(q_n\vt^n_0(a_l,z_f))\right]\nonumber \\
&\quad -\mathbb E\left[\ln M_{c,1}^{\RR}(q_n\vt^n_0(a_k,z_f))\right]\mathbb E\left[\ln M_{c,1}^{\RR}(q_n\vt^n_0(a_l,z_f))\right]\Biggr).
\end{align}
Letting now $n$ tend to infinity we obtain
\begin{align}
&C_1 \left(\mathbb E\left[M_{c,1}^{\RR,a_j} M_{c,1}^{\RR,a_l}\right]-\mathbb E\left[ M_{c,1}^{\RR,a_j}\right]\mathbb E\left[ M_{c,1}^{\RR,a_l}\right]\right).
\end{align}
Limit and integral are interchangeable because dominated convergence is applicable due to the boundedness of the logarithmic moment generating functions for $\vt\in (0,\vt_{**})$.
\qed\end{proof}
To complete the proof of Lemma \ref{R} we need to  prove 
tightness. To do so, we use, as usual,  the 
 Kolmogorov-Chentsov criterion \cite{KAL} and check the Conditions \ref{starttight} and \ref{versiontightness}. 
\begin{proof}[Proof of Lemma \ref{R}] 
The family of initial distributions is given by the random variables evaluated in $\vt^n_0(\e,z_f)$ for an $\e >0$ because  $a>n^{-1}\E^{\RR}[G_n(z_f)]>0$. This family is seen to be  tight using Chebychev's inequality:
\begin{align}
&  \P\left(\frac{1}{\sqrt n}\sum_{j=1}^{n_c}\left(\ln M_{c,j}^{\RR}(q_n\vt^n_0(\e,z_f))-\E\left[\ln M_{c,j}^{\RR}(q_n\vt^n_0(\e,z_f))\right]\right)\geq K\right)\nonumber \\
\leq& \frac{\tfrac 1n \sum_{j=1}^{n_c}\Va [\ln M_{c,j}^{\RR}(q_n\vt^n_0(\e,z_f))]}{K^2}=\frac{n_c}{n}\frac{\Va[\ln M^{\RR}_{c,1}(q_n\vt^n_0(\e,z_f))]}{K^2}.
\end{align}
With $\frac{n_c}{n}\to C_1$ and $q_n\to 1$ 
exist $\d,\wb{\d}$ and $n_0$ such that $\frac{n_c}{n}\leq C_1+\d$ and $q_n\vt^n_0(\e,z_f)\leq\vt^n_0(\e,z_f)+\wb{\d}$ for all $n\geq n_0$. Thus, 
\begin{equation}
\frac{n_c}{n}\Va[\ln M^{\RR}_{c,j}(q_n\vt^n_0(\e,z_f))]\leq(C_1+\d)\Va[\ln M_{c,1}^{\RR}(\vt^n_0(\e,z_f)+\wb{\d}))]
\end{equation}
 for all $n\geq n_0$ and we obtain 
\begin{align}\label{start1}
& \P \left(\frac{1}{\sqrt n}\sum_{j=1}^{n_c}(\ln M_{c,j}^{\RR}(q_n\vt^n_0(\e,z_f))-\E[\ln M_{c,j}^{\RR}(q_n\vt^n_0(\e,z_f))])\geq K\right) \nonumber \\
\leq &K^{-2}\max\left\lbrace\max_{i\in 1,\dots,n_0-1}\tfrac{n_c(i)}{i}\Va\left[Y_{\e,j}^i\right],(C_1+\d)\Va\left[\ln M_{c,1}^{\RR}(\vt^n_0(\e,z_f)+\wb{\d})\right]\right\rbrace.
\end{align}
For each $\wt{\e}$ we can choose $K$ large enough such that \ref{start1} $<\wt{\e}$ and thus we have proven tightness of the initial  distributions. 

It remains to check Condition \ref{versiontightness}. We have
\begin{align}
&\E\left[\left(\frac{d}{da}Z_{n,c}(q_n\vt^n_0(a,z_f))\right)^2\right]= \E\left[\left(\frac{1}{\sqrt n}\sum_{j=1}^{n_c}\frac{d}{da}Y_{a,j}^n\right)^2\right] \nonumber \\
= &{\frac 1n \left(\E\left[\sum_{j=1}^{n_c}\left(\frac{d}{da}Y_{a,j}^n\right)^2\right]+ \sum_{j=1}^{n_c}\sum_{i=1,i\neq j}^{n_c}\underbrace{\E\left[\tfrac{d}{da}Y_{a,j}^n\tfrac{d}{da}Y_{a,i}^n\right]}_{=0}\right) =\frac{n_c}{n} \E\left[\left(\frac{d}{da}Y_{a,j}^n\right)^2\right]}
\end{align}
because $Y_{a,j}^n$ and thus $\frac{d}{da}Y_{a,j}^n$ are centered  i.i.d. random variables. They are independent for different $j$ because $\vt^n_0(a,z_f)$ depends only on the random variable $W_f$ by definition. Thus, it is enough to show that $\E\left[(\frac{d}{da}Y_{a,j}^n)^2\right]$ is bounded. Since $\E[X]$ and thus $X-\E[X]$ are bounded if the random variable $X$ is bounded it suffices in our scenario to show boundedness of the uncentered random variable. 
Let $X_{a,j}^n\equiv \vt^n_0(a,z_f)Z_j^cW_j$. Then
\begin{align}
0\leq&\frac{d}{da}\left(\ln \E^{\RR}\left[e^{q_nX_{a,j}^n}\right]\right)=\frac{\E^{\RR}\left[q_n\frac{d}{da}X_{a,j}^ne^{q_nX_{a,j}^n}\right]}{\E^{\RR}\left[e^{q_nX_{a,j}^n}\right]}\nonumber \\
\leq &q_n\max_a\left(\frac{d}{da}X_{a,j}^n\right) \frac{\E^{\RR}\left[e^{q_nX_{a,j}^n}\right]}{\E^{\RR}\left[e^{q_nX_{a,j}^n}\right]} \leq q_nZ_1^{c,\max}W_1^{\max}\max_a \frac{d}{da}\left(\vt^n_0(a,z_f)\right).
\end{align}
We know by an application of the implicit function theorem that 
\begin{equation}
\frac {d}{da}\vt^n_0(a,z_f)=\left(g_n''(q_n\vt^n_0(a,z_f))\right)^{-1}.
\end{equation}
Thus, we get tightness if $g_n''(q_n\vt^n_0(a,z_f))>C>0$.
\qed\end{proof}
With analogous calculations we get the convergence  of the derivatives of $Z_n(q_n\vt^n_0(a,z_f))$.
\begin{lemma}
	The processes $Z_n'(q_n\vt^n_0(a,z_f))$ and $Z_n''(q_n\vt^n_0(a,z_f))$ as processes on the Wiener Space with parameter $a$ converge weakly if there exists $C>0$ such that $g_n''(q_n\vt^n_0(a,z_f))>C>0$.
\end{lemma}
\begin{proof}
In analogy to the previous proof we define 
\begin{equation}
 \left(Y_{a,j}^n\right)'= \frac{\E^{\RR}\left[q_nZ_j^cW_je^{ q_nX_{a,j}^n}\right]}{\E^{\RR}\left[e^{q_nX_{a,j}^n}\right]}-\E\left[\frac{\E^{\RR}\left[q_nZ_j^cW_je^{ q_nX_{a,j}^n}\right]}{\E^{\RR}\left[e^{q_nX_{a,j}^n}\right]}\right]
\end{equation}
 and 
\begin{align}  
\left(Y_{a,j}^n\right)'' =  &\frac{\E^{\RR}\left[(q_nZ_j^cW_j)^2e^{q_nX_{a,j}^n}\right]}{\E^{\RR}\left[e^{ q_nX_{a,j}^n}\right]}-\left(\frac{\E^{\RR}\left[q_nZ_j^cW_je^{q_nX_{a,j}^n}\right]}{\E^{\RR}\left[e^{q_nX_{a,j}^n}\right]}\right)^2 \nonumber \\
& - \E\left[\frac{\E^{\RR}\left[(q_nZ_j^cW_j)^2e^{q_nX_{a,j}^n}\right]}{\E^{\RR}\left[e^{ q_nX_{a,j}^n}\right]}-\left(\frac{\E^{\RR}\left[q_nZ_j^cW_je^{q_nX_{a,j}^n}\right]}{\E^{\RR}\left[e^{q_nX_{a,j}^n}\right]}\right)^2\right].
\end{align}
The constitutive parts of the processes under consideration are given by
\begin{equation}
Z_{n,c}'(q_n\vt^n_0(a,z_f))=\frac{1}{\sqrt n}\sum_{j=1}^{n_c} \left(Y_{a,j}^n\right)'
\quad
\text{ and }
\quad
Z_{n,c}''(q_n\vt^n_0(a,z_f))= \frac{1}{\sqrt n}\sum\limits_{j=1}^{n_c} \left(Y_{a,j}^n\right)''.
\end{equation}
With the notation $(\chi_j^n)'=\frac {1}{\sqrt n}((Y_{a_1,j}^n)',\dots,(Y_{a_m,j}^n)')$ and $(\chi_j^n)''=\frac {1}{\sqrt n}((Y_{a_1,j}^n)'',\dots,(Y_{a_m,j}^n)'')$
we obtain $|(\chi_j^n)'|^2\leq \frac{4m}{n}(Z_1^{c,\max}W_1^{\max})^2$ and $|(\chi_j^n)''|^2\leq \frac{4m}{n}(Z_1^{c,\max}W_1^{\max})^4$. Hence, the convergence of the finite dimensional distributions of $Z_{n,c}'$ and $Z_{n,c}''$ follows with the same argument as before.

Concerning the tightness of the initial distributions we get the two following bounds using again Chebychev's inequality: 
\begin{align}
&\P \left(\frac{1}{\sqrt n}\sum_{j=1}^{n_c}(Y_{\e,j}^n)'\geq K\right) \nonumber \\
\leq &  K^{-2}\max\left \lbrace\max_{i\in 1,\dots,n_0-1} \tfrac{n_c(i)}{i} \Va\left[(Y_{\e,j}^i)'\right],(C_1+\d)\Va\left[\frac{\E^{\RR}[Z_1^cW_1e^{(\vt^n_0(\e,z_f)+\wb{\d}) Z_1^cW_1}]}{\E^{\RR}[e^{(\vt_0(\e,z_f)+\wb{\d})Z_1^cW_1}]}\right]\right \rbrace
\end{align}
and
\begin{align}
& \P \left(\frac{1}{\sqrt n}\sum_{j=1}^{n_c}(Y_{\e,j}^n)''\geq K\right) \leq K^{-2}\max \left\lbrace \max_{i\in 1,\dots,n_0-1} \frac{n_c(i)}{i}\Va[(Y_{\e,j}^i)''],\right.\nonumber \\
&  \quad \left.(C_1+\d)\Va\left[\frac{\E^{\RR}[(Z_1^cW_1)^2e^{(\vt^n_0(\e,z_f)+\wb{\d}) Z_1^cW_1}]}{\E^{\RR}[e^{(\vt^n_0(\e,z_f)+\wb{\d}) Z_1^cW_1}]}-\left(\frac{\E^{\RR}[Z_1^cW_1e^{(\vt^n_0(\e,z_f)+\wb{\d}) Z_1^cW_1}]}{\E^{\RR}[e^{(\vt^n_0(\e,z_f)+\wb{\d}) Z_1^cW_1}]}\right)^2\right]\right\rbrace.
\end{align}
Using again Condition \ref{versiontightness} it is enough to bound 
\begin{equation}
\frac{d}{da}\frac{\E^{\RR}[q_nZ_j^cW_je^{q_nX_{a,j}^n}]}{\E^{\RR}[e^{q_nX_{a,j}^n}]}
\end{equation}
and
\begin{equation}
\frac{d}{da}\left(\frac{\E^{\RR}[(q_nZ_j^cW_j)^2e^{q_nX_{a,j}^n}]}{\E^{\RR}[e^{q_nX_{a,j}^n}]}-\left(\frac{\E^{\RR}[q_nZ_j^cW_je^{q_nX_{a,j}^n}]}{\E^{\RR}[e^{q_nX_{a,j}^n}]}\right)^2\right).
\end{equation}
These derivatives are given by
\begin{equation}\label{der1}
\frac{\E^{\RR}[(\vt^n_0)'(a,z_f)(q_nZ_j^cW_j)^2e^{q_nX_{a,j}^n}]}{\E^{\RR}[e^{q_nX_{a,j}^n}]}-\frac{\E^{\RR}[(\vt^n_0)'(a,z_f)q_nZ_j^cW_je^{q_nX_{a,j}^n}]\E^{\RR}[q_nZ_j^cW_je^{q_nX_{a,j}^n}]}{(\E^{\RR}[e^{q_nX_{a,j}^n}])^2}
\end{equation}
and
\begin{align}\label{second1}
&\frac{\E^{\RR}[(\vt^n_0)'(a,z_f) (q_nZ_j^cW_j)^3e^{q_nX_{a,j}^n}]}{\E^{\RR}[e^{q_nX_{a,j}^n}]}-\frac{\E^{\RR}[(q_nZ_j^cW_j)^2e^{q_nX_{a,j}^n}]\E^{\RR}[(\vt^n_0)'(a,z_f)q_nZ_j^cW_je^{q_nX_{a,j}^n}]}{(\E^{\RR}[e^{q_nX_{a,j}^n}])^2}\nonumber \\
& - 2\frac{\E^{\RR}[(q_nZ_j^cW_j)^2(\vt^n_0)'(a,z_f)e^{q_nX_{a,j}^n}]\E^{\RR}[q_nZ_j^cW_je^{q_nX_{a,j}^n}]}{(\E^{\RR}[e^{q_nX_{a,j}^n}])^2}\nonumber \\
& +2\frac{\E^{\RR}[(\vt^n_0)'(a,z_f)q_nZ_j^cW_je^{q_nX_{a,j}^n}](\E^{\RR}[q_nZ_j^cW_je^{q_nX_{a,j}^n}])^2}{(\E^{\RR}[e^{q_nX_{a,j}^n}])^3}.
\end{align}
We are able to bound the first derivative according to 
\begin{align}
&-\max_{a}(\vt^n_0)'(a,z_f)(Z_j^{c,\max}W_j^{\max})^2
\leq \ref{der1}
\leq \max_{a}(\vt^n_0)'(a,z_f)(Z_j^{c,\max}W_j^{\max})^2 .
\end{align}
Thus, we get the same criterion as in Lemma \ref{R} for boundedness and thus tightness here. We can bound each summand in \ref{second1} very similar to the previous cases and end up with the bound $4\max_{a}(\vt^n_0)'(a,z_f)(Z_j^{c,\max}W_j^{\max})^3$. Thus, we again arrive at the same criterion to get tightness.
\qed\end{proof}

\begin{lemma}
$X_a^n=(Z_n(q_n\vt^n_0(a,z_f)),Z_n'(q_n\vt^n_0(a,z_f)),Z_n''(q_n\vt^n_0(a,z_f)))$ converges weakly if there exists $C>0$ such that $g_n''(q_n\vt^n_0(a,z_f))>C>0$.
\end{lemma}
\begin{proof}
The structure of the proof is the same. First, we consider the finite dimensional distributions. Let $Y_{a,j}^n, (Y_{a,j}^n)'$ and $(Y_{a,j}^n)''$ be  defined as above. 
We investigate now
\begin{equation}
 \chi_j^n\equiv \tfrac{1}{\sqrt n}(Y_{a_1,j}^n,(Y_{a_1,j}^n)',(Y_{a_1,j}^n)'',\dots,Y_{a_l,j}^n,(Y_{a_l,j}^n)',(Y_{a_l,j}^n)'').
 \end{equation}
These vectors are again independent for different $j$. The boundedness of $|\chi_j^n|^2$ follows directly by the boundedness in the previous cases. Again, $|\chi_j^n|$ tends to $0$ such that the Lindeberg condition is satisfied. Part 1 of Condition \ref{A} holds because we consider centered random variables. Part 2  holds because we can bound $|\chi_j^n|^2$.  
It can be shown that the initial distributions are tight using again  Chebychev's inequality. To prove tightness we have to show that
\begin{align}
	&\E\left[|X_{a+h}^n-X_a^n|^2\right]\nonumber \\
	= &\E\left[(Z_n(q_n\vt^n_0(a+h))-Z_n(q_n\vt^n_0(a,z_f)))^2+(Z_n'(q_n\vt^n_0(a+h))-Z_n'(q_n\vt^n_0(a,z_f)))^2\right. \nonumber \\
	&\left.
	 +(Z_n''(q_n\vt^n_0(a+h))-Z_n''(q_n\vt^n_0(a,z_f)))^2\right]\leq C|h|^2
\end{align}
This holds true because we have already seen that each summand can be bounded by the right-hand side for a certain C. Thus, we have just to sum up the different constants.
\qed\end{proof}
Finally, we come to the proof of the result about the rate function.
\begin{proof}[Proof of Theorem \ref{ratefct}]
We look at the $\vt_n(a,z_f)$ determining equation 
\begin{align}
a-\frac{z_f}{n}W_f &=q_ng_n'(q_n\vt) + \frac{1}{\sqrt n} q_nZ_n'(q_n\vt).  
\end{align}
and write the solution of this equation in the form $\vt^n_0(a,z_f)+\d^n(a,z_f)$, where $\vt^n_0(a,z_f)$  is defined as the solution of 
\begin{equation}
a-\frac{z_f}{n}W_f = q_ng_n'(q_n\vt).
\end{equation}
$\d^n(a,z_f)$ denotes the stochastic perturbation of this equation caused by the process $Z_n$.
By definition of $\vt^n_0(a,z_f)$, we have $q_ng_n'(q_n\vt^n_0(a,z_f))=a-z_f/nW_f$. We can derive an expression for $\d^n(a,z_f)$ using a first order Taylor expansion. To keep the notation short we drop the arguments and write just $\vt^n_0$ and $\d^n$. 
We obtain
\begin{align}
&a-\frac{z_f}{n}W_f= q_n\left[g_n'(q_n(\vt^n_0+\d^n)) + \frac{1}{\sqrt n }Z_n'(q_n(\vt^n_0+\d^n))\right]  \nonumber \\
\Leftrightarrow \quad&a-\frac{z_f}{n}W_f=q_n\left[g_n'(q_n\vt^n_0)+q_n\d^ng_n''(\vt^n_0) +\frac{1}{\sqrt n} Z_n'(q_n\vt^n_0) +q_n\d^n\frac{1}{\sqrt n}Z_n''(q_n\vt^n_0) + \po(\d^n)\right]\nonumber \\
\Leftrightarrow \quad&0=q_n^2\d^n\left(g_n''(q_n\vt^n_0)+\frac{1}{\sqrt n} Z_n''(q_n\vt^n_0)\right)+ \frac{q_n}{\sqrt n}Z_n'(q_n\vt^n_0) + \po(\d^n)\nonumber \\
\Leftrightarrow \quad &\d^n = \frac{-\frac{1}{\sqrt n}Z_n'(q_n\vt^n_0)+\po(\d^n)}{q_n(g_n''(q_n\vt^n_0)+\frac{1}{\sqrt n}Z_n''(q_n\vt^n_0))}= \frac{-\frac{1}{\sqrt n}Z_n'(q_n\vt^n_0)}{q_n(g_n''(q_n\vt^n_0)+\frac{1}{\sqrt n}Z_n''(q_n\vt^n_0))}+\po(\d^n).
\end{align}
The rate function can be rewritten as 

\begin{align}
I_n^{\RR}(a,z_f) &=a\vt_n(a,z_f)-\Psi_n^{\RR}(\vt_n(a,z_f)) \nonumber \\
&= \left(a-\frac{z_f}{n}W_f\right)(\vt^n_0+\d^n)-g_n(q_n(\vt^n_0+\d^n)) -\frac{1}{\sqrt n}Z_n(q_n(\vt^n_0+\d^n)).
\end{align}
A second order Taylor expansion and reordering of the involved terms yields
\begin{align}\label{I1}
I_n^{\RR}(a,z_f)=& \underbrace{a\vt^n_0-g_n(q_n\vt^n_0)}_{{=:}I^n_0(a,z_f)} -\frac{z_f}{n}W_f\vt_0^n-\frac{1}{\sqrt n}Z_n(q_n\vt^n_0)+\underbrace{\left(\left(a-\frac{z_f}{n}W_f\right)-q_ng_n'(q_n\vt^n_0)\right)}_{=0}\d^n
\nonumber \\
&\quad -\frac{1}{\sqrt n}q_n\d^n Z_n'(q_n\vt^n_0) -\frac 12 (q_n\d^n)^2\left(g_n''(q_n\vt^n_0)+\frac{1}{\sqrt n}Z_n''(q_n\vt^n_0)\right) + \po((q_n\d^n)^2)
\end{align}
The stochastic process $Z_n(q_n\vt^n_0(a,z_f))$ converges weakly to the mentioned Gaussian process according to Lemma \ref{R}.  $g_n''$ is of the order $\OO(1)$ according to 
\begin{align}
 g_n''(q_n\vt)&\to C_1\E\left[\frac{\E^{\RR}[(Z_1^cW_1)^2e^{\vt Z_1^cW_1})]\E^{\RR}[e^{\vt Z_1^cW_1}]-\E^{\RR}[Z_1^cW_1e^{\vt Z_1^cW_1}]^2}{(\E^{\RR}[e^{\vt Z_1^cW_1}])^2}\right] \nonumber \\
& \quad +  C_2\E\left[\frac{\E^{\RR}[(Z_1^vW_1)^2e^{\vt Z_1^vW_1}]\E^{\RR}[e^{\vt Z_1^vW_1}]-\E^{\RR}[Z_1^vW_1e^{\vt Z_1^vW_1}]^2}{(\E^{\RR}[e^{\vt Z_1^vW_1}])^2}\right]
\end{align}
Together with the joint weak convergence of the processes $Z_n'(q_n\vt^n_0(a,z_f))$ and $Z_n''(q_n\vt^n_0(a,z_f))$ this yields $\d^n\in\OO(1/\sqrt n)$.
Furthermore, this implies $f \in \po(1/\sqrt n)$ for each $f \in \po(\d^n)$  and $f\in \po(1/n)$ for each $f\in \po((\d^n)^2)$.  
We plug in Equation \ref{I1} the expression for $\d^n$  and obtain 
\begin{align}
I_n^{\RR}(a,z_f)&= I^n_0(a,z_f)-\frac{z_f}{n}W_f\vt_0^n -\frac{1}{\sqrt n}Z_n(q_n\vt^n_0) \nonumber \\
&\quad - \frac{q_n^2}{2}\left(\frac{-(\frac{1}{\sqrt n}Z_n'(q_n\vt^n_0))}{q_n(g_n''(q_n\vt^n_0)+\frac{1}{\sqrt n}Z_n''(q_n\vt^n_0))}+\po(\d^n)\right)^2\left(g_n''(q_n\vt^n_0)+\frac{1}{\sqrt n}Z_n''(q_n\vt^n_0)\right)  \nonumber \\
&\quad -\frac{q_n}{\sqrt n}Z_n'(q_n\vt^n_0)\left( \frac{-\frac{1}{\sqrt n}Z_n'(q_n\vt^n_0)}{q_n(g_n''(q_n\vt^n_0)+\frac{1}{\sqrt n}Z_n''(q_n\vt^n_0))}+\po(\d^n)\right) +\po((\d^n)^2)\nonumber \\
&= I^n_0(a,z_f)-\frac{z_f}{n}W_f\vt_0^n -\frac{1}{\sqrt n}Z_n(q_n\vt^n_0)
- \frac{q_n^2}{2}\left(g_n''(q_n\vt_0^n)+\frac{1}{\sqrt n}Z_n''(q_n\vt_0^n)\right)\nonumber\\
&\quad \times \left[ 
\frac{\frac 1n(Z_n'(q_n\vt_0^n))^2}{q_n^2\left(g_n''(q_n\vt_0^n)+\frac{1}{\sqrt n}Z_n''(q_n\vt_0^n)\right)^2}-\frac{\frac{1}{\sqrt n}Z_n'(q_n\vt_0^n)\po(\d^n)}{q_n\left(g_n''(q_n\vt_0^n)+\frac{1}{\sqrt n}Z_n''(q_n\vt_0^n)\right)}+\po((\d^n)^2)\right]\nonumber \\
&\quad +\frac{\frac{1}{n}(Z_n'(q_n\vt_0^n))^2}{g_n''(q_n\vt_0^n)+\frac{1}{\sqrt n}Z_n''(q_n\vt_0^n)}+ \frac{1}{\sqrt n}Z_n'(q_n\vt_0^n)\po(\d^n)+\po((\d^n)^2).
\end{align}
According to the observations concerning $\d^n$ this equals
\begin{align}
I^n_0(a,z_f)-\frac{z_f}{n}W_f\vt_0^n -\frac{1}{\sqrt n}Z_n(q_n\vt^n_0) + \frac{\frac{1}{2n}(Z_n'(q_n\vt_0^n))^2}{g_n''(q_n\vt_0^n)+\frac{1}{\sqrt n}Z_n''(q_n\vt_0^n)} +\po\left(\frac 1n\right),
\end{align}
where $\frac{(Z_n'(q_n\vt_0^n))^2}{g_n''(q_n\vt_0^n)+\frac{1}{\sqrt n}Z_n''(q_n\vt_0^n)}$ converges weakly due to continuous mapping and the joint weak convergence of $Z_n'(q_n\vt_0^n)$ and $Z_n''(q_n\vt_0^n)$.
This completes the proof of the theorem.
\qed\end{proof}

\paragraph{Case 3.)}
Conditioning  on $\ZZ$ produces again i.i.d. random variables because $Z_j$ are measurable w.r.t. $\ZZ$ and the stimulation rates are independent of this $\s$-algebra.
\begin{theorem}\label{VorBedZ}
 Let $(a_n)_{n\in\mathbb N}$ be defined by $a_n \equiv a$ and $g_{act}(n)=an$ such that $g_{act}(n) >\E^{\ZZ}(G_n(z_f))$ and $a < \sup_{\vt \in \mathbb R} \frac{d}{d\vt} \Psi_n^{\ZZ}(\vt)$  for all $n\in \mathbb N$. Then Theorem \ref{abwthm} is almost surely applicable provided $z_f/n\downarrow 0$, the distribution functions of the stimulation rates are neither lattice valued nor concentrated on one point and the moment generating functions $M_{c,j}^{\ZZ}(\vt)$, $M_{v,j}^{\ZZ}(\vt)$ and $M^{\ZZ}(\vt)$ as well as $M_c(\vt)$, $M_v(\vt)$ and $M(\vt)$ are finite for each $\vt\in\mathbb R$. The rate function is
 \begin{align}
 	I_n^{\ZZ}(a,z_f) =   &a\vt_n(a,z_f) -\frac 1n \left(\sum_{j=1}^{n_c}\ln M_{c,j}^{\ZZ}(q_n\vt_n(a,z_f)) \right.\nonumber \\
	&\left. + \sum_{j=n_c+1}^{n_c+n_v} \ln M_{v,j}^{\ZZ}(q_n\vt_n(a,z_f)) + \ln M(z_f\vt_n(a,z_f))\right). 
 \end{align}
\end{theorem}
\begin{proof} The proof of this theorem goes along the same lines as the analogous result in Case 2.) and will be skipped.
\qed\end{proof}

\paragraph{Investigation of the rate function.}

In this case, the properties of the large deviation rate function can again be described  by a functional central limit theorem. Using the notation from Section \ref{results} we obtain the following result
\begin{theorem}\label{ratefct}
If there exists a constant $C$ such that $\wt{g}_n''(q_n\wt{\vt}^n(a,z_f))>C>0$, the rate function takes the form 
 \begin{align}
I_n^{\ZZ}(a,z_f) = &\wt{I}^n(a,z_f) -\frac{1}{\sqrt n}\wt{Z}_n(q_n\wt{\vt}^n(a,z_f))-\frac{1}{n}\ln M(z_f\wt{\vt}^n(a,z_f)) + R_n ,
\end{align}
where $R_n\in \OO\left(\frac1n\right)$.
 $\wt{Z}_n(q_n\wt{\vt}^n(a,z_f))$
converges weakly to the Gaussian process $Z_{a} + \overline{Z}_{a}$. $Z_{a}$ and $\overline{Z}_{a}$ are both Gaussian processes with expectation functions $\E[Z_{a}]=0=\E[\overline{Z}_{a}]$  and covariance functions
\begin{align}
\Cov (Z_{a},Z_{a '}) = C&\left( \E[\ln (\E^{\ZZ}[e^{\wt{\vt}_0(a) Z_1^cW_1}])\ln (\E^{\ZZ}[e^{\wt{\vt}_0(a ') Z_1^cW_1}])]\right.\nonumber \\
&\left.-\E[\ln(\E^{\ZZ}[e^{\wt{\vt}_0(a) Z_1^cW_1}])]\E[\ln(\E^{\ZZ}[e^{\wt{\vt}_0(a ') Z_1^cW_1}])]\right)
\end{align}
and
\begin{align}
\Cov (\overline{Z}_{a},\overline{Z}_{a '}) = \widetilde{C}&\left( \E[\ln (\E^{\ZZ}[e^{\wt{\vt}_0(a) Z_1^vW_1}])\ln( \E^{\ZZ}[e^{\wt{\vt}_0(a ') Z_1^vW_1}])]\right. \nonumber \\
&\left.-\E[\ln(\E^{\ZZ}[e^{\wt{\vt}_0(a) Z_1^vW_1}])]\E[\ln(\E^{\ZZ}[e^{\wt{\vt}_0(a ') Z_1^vW_1}])]\right).
\end{align}
\end{theorem}
The approach to prove this result is the same as in Case 2.)  and we will 
not present the details.

\section{Conclusion and Outlook}

The main new aspects of the present work are the investigation of the conditional scenarios and the establishment of a higher robustness of the model using classes of distributions instead of concrete distributions.

The first point allows a more precise understanding and interpretation of the activation mechanism. The presented results show that the parameter $a$ can be chosen such that the activation probabilities increase exponentially with $z_f$ for the regime $\sqrt n \ll z_f \ll n$ in Cases 1.) and 3.). For Case 2.) this result depends on the actual value of the stimulation rate of the foreign peptide, $W_f$. This value has to be large enough in order that the probability of activation increases exponentially with $z_f$. In biological terms this means that the over all frequency of T-Cell activation increases as desired. This growth is caused by those T-Cell types which interact strongly with the given foreign peptide type. These findings are on the one hand very suitably captured by the title of recent work of van den Berg et al. \cite{BMS} "Specific T-cell activation in an unspecific T-cell repertoire" and on the other hand justify this formulation. This interpretation also suits to the histograms of stimulation rates and the explanation of these histograms in \cite{BL}.

The generalized distribution assumptions allow the choice of different distributions depending on the mechanisms that should be included into the model.  For example MHC-loading fluctuations, influence of different affinity of the different peptide types to the same receptor and maybe co-stimulation could be considered. 
For the relevance of co-stimulation in the immune response see, e.g. the recent  review by Chen
\cite{chen2013molecular}.

One aim of future work is to investigate the mechanism of \emph{negative selection}. Thereby T-Cells which interact too strongly with the body's own structures are deleted. Van den Berg et al. investigated this mechanism in a different model setting in \cite{BM} and also Zint et al. \cite{BZH} included negative selection into numerical simulations. In this case the situation becomes mathematically more involved, since the selection causes dependencies in between the stimulation rates as was already pointed out in \cite{HU}.
\nocite{Ja}
\nocite{JI}

\begin{acknowledgement} We are very grateful to Ellen Baake and Frank den 
Hollander for numerous discussions and valuable input. A.B. is partially supported through the German Research Foundation in the Priority Programme 1590 ``Probabilistic Structures in Evolution'' 
 and
the Hausdorff Center for Mathematics (HCM).  He is member of the Cluster of Excellence ``ImmunoSensation'' at Bonn University.
H.M. is supported by the German Research Foundation in the Bonn International Graduate School in Mathematics (BIGS). 
\end{acknowledgement}

\bibliographystyle{abbrv}
\bibliography{mybib}

\end{document}